%% file: main.tex
\documentclass[a4paper, english]{article}

\input{preamble}
\usepackage{authblk}

\usepackage[left=2cm,top=3cm,right=2cm,bottom=3cm]{geometry}
\title{Algorithmic correspondence and analytic rules}

\author{Andrea De Domenico}
\author{Giuseppe Greco}
\author{Alessandra Palmigiano}
\affil{Vrije Universiteit, Amsterdam}

\date{}

\begin{document}

\maketitle

\begin{abstract}
We introduce the algorithm \Algname which takes classical modal formulas in input, and, when successful, effectively generates: (a) (analytic) geometric rules of the labelled calculus G3K, and (b) cut-free derivations (of a certain `canonical' shape) of each given input formula in the geometric labelled calculus obtained by adding the  rule in output to G3K. We show that \Algname successfully terminates whenever its input formula is a (definite) analytic inductive formula, in which case, the geometric axiom corresponding to the output rule is, modulo logical equivalence, the first-order correspondent of the input formula.  In proving the correctness of MASSA, we also show that the algorithm for the elimination of second-order quantifiers SCAN is complete with respect to the class of inductive analytic formulas. Finally, we show how our algorithm can be extended to the class of inductive formulas and to modal logic with quantifiers.
\end{abstract}

\tableofcontents

\input{paper_MASSA/Introduction}

\input{paper_MASSA/LabelledRelationalSequentCalculi}
\input{paper_MASSA/AnalyticInductive}
\input{paper_MASSA/FirstOrderCorrespondentsAsInferenceRules}
\input{paper_MASSA/NewDevelopments}
\input{paper_MASSA/Conclusions}

\bibliographystyle{abbrv}
\bibliography{reference}

\end{document}

%% file: preamble.tex
\usepackage[utf8]{inputenc}
\usepackage{pdfpages}
\usepackage{hyperref}
\usepackage{bussproofs}
\usepackage{amssymb}
\usepackage{amsthm}
\usepackage{marginnote}
\usepackage{xcolor}
\usepackage{mathtools}
\usepackage{algpseudocode}
\usepackage{algorithm}
\usepackage{xspace}
\usepackage{forest}
\usepackage{multirow}
\usepackage{soul}
\usepackage{cmll}
\usepackage{mathdots}
\usepackage{tikz}
\usetikzlibrary{arrows,fit}

\usepackage{epigraph}

\usepackage{imakeidx}
\makeindex

\DeclareRobustCommand % Copied for comparison with answer by Phelype Oleinik
\Compactldots{\mathinner{\ldotp\mkern-3mu\ldotp\mkern-3mu\ldotp}}
\renewcommand{\ldots}{%
  \Compactldots
}

\newtheorem{thm}{Theorem}[section]
\newtheorem{theorem}{Theorem}[section]

\newtheorem{corollary}[thm]{Corollary}

\newtheorem{proposition}[thm]{Proposition}

\newtheorem{lemma}[thm]{Lemma}
\newtheorem{notation}[thm]{Notation}

\newtheorem{definition}[thm]{Definition}

\newtheorem{example}[thm]{Example}

\newtheorem{remark}[thm]{Remark}

\newcommand{\vd}{\ \,\textcolor{black}{\vdash}\ \,}
\newcommand{\ol}[1]{\overline{#1}}
\EnableBpAbbreviations
\newcommand{\aand}{\ensuremath{\wedge}\xspace}
\newcommand{\aor}{\ensuremath{\vee}\xspace}
\newcommand{\ararr}{\ensuremath{\rightarrow}\xspace}

\newcommand{\fns}{\footnotesize}

\def\mc{\multicolumn}

\newcommand{\aatop}{\ensuremath{\top}\xspace}
\newcommand{\abot}{\ensuremath{\bot}\xspace}

\newcommand{\wdia}{\ensuremath{\Diamond}\xspace}

\newcommand{\GTK}{G3K }

\newcommand{\Algname}{MASSA }

\newcommand{\nx}{\textcolor{black}{x}}

\newcommand{\NR}{\textcolor{black}{R}}

\newcommand{\ny}{\textcolor{black}{y}}

\newcommand{\wbox}{\ensuremath{\Box}\xspace}

\newcommand{\bvarphi}{\textcolor{blue}{\varphi^b}}
\newcommand{\bA}{\textcolor{blue}{A^b}}
\newcommand{\bAu}{\textcolor{blue}{A^b_1}}
\newcommand{\bAd}{\textcolor{blue}{A^b_2}}

\newcommand{\bB}{\textcolor{blue}{B^b}}
\newcommand{\bp}{\textcolor{blue}{p^b}}

\newcommand{\rvarphi}{\textcolor{red}{\varphi^r}}
\newcommand{\rA}{\textcolor{red}{A^r}}

\newcommand{\rAt}{\textcolor{red}{A^r_3}}
\newcommand{\rAq}{\textcolor{red}{A^r_4}}
\newcommand{\rB}{\textcolor{red}{B^r}}
\newcommand{\rp}{\textcolor{red}{p^r}}

\newcommand{\bgamma}{\textcolor{blue}{\gamma}}

\newcommand{\rdelta}{\textcolor{red}{\delta}}

\def\fCenter{{\mbox{$\ \vdash\ $}}}

\usepackage{forest}

\tikzset{
	treenode/.style = {align=center, inner sep=0pt, text centered},
	Ske/.style = {treenode, ellipse, double, draw=black,
		minimum width=6pt, thick},% arbre rouge noir, noeud noir
	PIA/.style = {treenode, ellipse, black, draw=black,
		minimum width=6pt},% arbre rouge noir, noeud rouge
	Crit/.style = {treenode, rectangle, draw=black,
		minimum width=0.5em, minimum height=0.5em}% arbre rouge noir, nil
}

\usepackage{txfonts}

\usepackage{enumerate}

\usetikzlibrary{patterns}

%% file: paper_MASSA/Introduction.tex
\section{Introduction}
\label{sec:Introduction_MASSA}

The labelled calculus \GTK was presented by Sara Negri in \cite{negri2005proof} as a basic   G3-style sequent calculus for the normal modal logic $K$ (see \cite[Chapter 3]{Negri_van_Plato} and \cite[Chapter 11]{negri2011proof} for the genesis of this calculus). The calculus \GTK shares many of the characteristic properties of Gentzen's original sequent calculus G3 for classical logic; for instance, all its rules are invertible, and the basic structural rules (weakening, contraction and cut) are admissible. Moreover, in \cite{negri2005proof}, Negri introduces a general method for  extending \GTK so as to capture a large class of axiomatic extensions of $K$; namely,   all those axiomatic extensions of $K$ which define  {\em elementary} (i.e.~first-order definable) classes of Kripke frames, and such that their defining  first-order conditions are, modulo logical equivalence,    geometric implications. The rules generated by Negri's  method for capturing these axiomatic extensions of $K$ are defined on the basis of their corresponding geometric implications, and are referred to as {\em geometric rules}. Negri uniformly shows that the structural rules (and cut in particular) are  admissible in the calculi obtained by extending \GTK with geometric rules. 

One important subclass of geometric implications is given, modulo logical equivalence, by  the first-order correspondents of the class of  {\em analytic inductive formulas} in classical modal logic. General (i.e.~not necessarily analytic) inductive formulas have been introduced by Goranko and Vakarelov in \cite{Goranko:Vakarelov:2006}, and have been shown to have  (local) first-order correspondents, which can be effectively computed via an algorithmic correspondence procedure introduced in \cite{Conradie:et:al:SQEMAI}.

In the present paper, we refine Negri's method for extending G3K, and introduce the  algorithm \Algname (Minimal ASSumption Algorithm) for generating analytic labelled rules uniformly and equivalently capturing the {\em analytic inductive} axiomatic extensions of $K$. An important difference between the algorithmic rule-generation method introduced in this paper and Negri's method is that the present method takes {\em modal formulas} in input, and, if the input formula is analytic inductive (cf.~Section \ref{sec:AnalyticInductiveFormulas}), it computes its equivalent analytic rule {\em directly} from the input formula, via a computation which incorporates the effective generation of its first-order correspondent, whereas Negri's method starts from geometric implications in the first-order frame correspondence language, and generates rules which are equivalent to those modal formulas which are assumed to have a first-order correspondent which is (logically equivalent to) a geometric implication.

Ultimately, modal correspondence theory is concerned with the elimination of certain second-order monadic quantifiers. The soundness proof of MASSA exploits the algorithm for the elimination of second-order quantifiers SCAN \cite{GaOh92c}, invented by Dov M. Gabbay and Hans J\"urgen Ohlbach.

\paragraph{Structure of the chapter.} This paper is structured as follows. In Section \ref{sec: prelim_MASSA}, we collect basic definitions and results on \GTK and analytic inductive formulas in classical modal logic;  
in Section \ref{sec:algorithm_MASSA}, we introduce the algorithm \Algname and provide intuitive motivation for some of its key steps.
In Section \ref{sec: examples_MASSA},  we illustrate how \Algname works, by running it on some well known modal axioms and more useful examples; in Section \ref{sec:SCAN_successful}, we prove that SCAN is complete with respect to the class of analytic inductive formulas; in Section \ref{exe:MASSA_termination_auxiliaries_definitions}, we prove the soundness and termination of the algorithm \Algname leveraging the soundness of SCAN; in Section \ref{sec:extending_MASSA}, we illustrate how to extend MASSA to more general settings; we conclude in Section \ref{sec: conclusions}.

%% file: paper_MASSA/LabelledRelationalSequentCalculi.tex
%\section{Generalized Gentzen calculi}
%\label{sec:GeneralizedGentzenCalcul}
%
\section{Preliminaries}
\label{sec: prelim_MASSA}

\subsection{The labelled  calculus \GTK}
\label{sec:LabelledRelationalSequentCalculi}

In what follows, we adopt the usual conventions: $p, q, \ldots$ denote proposition variables, $x, y, z, \ldots$ are labels (corresponding to world-variables in the intended interpretation on Kripke frames), given a label $x$ and a modal formula $A$, well-formed formulas are of the type $x: A$, while $\varphi, \psi, \ldots$ are meta-variables for well-formed formulas. $\Gamma, \Delta, \ldots$ are meta-variables for sets of wffs, and a sequent is an expression of the form $\Gamma \vd \Delta$. Given a sequent $S = \Gamma \vd \Delta$, if the formula $\varphi \in \Gamma$ (resp.~$\varphi \in \Delta$), we say that $\varphi$ occurs in \emph{precedent} (resp.~\emph{succedent}) position in $S$.

%Here we also explicitly distinguish between "," and ";". The reading of the symbol "," is contextual (as usual in this literature): it is interpreted as $\aand$ when occurring in precedent position and as $\aor$ when occurring in succedent position. The symbol ";" is interpreted as $\aand$ (both when it occurs in precedent and in succedent position). Indeed, the rule Id$_{xRy}$ can be omitted and the rules $\wbox_L$ and $\wdia_R$ can be substituted by binary rules (see \cite{xxx}). \blue{ADD: Probably we should eliminate rule Id$_{xRy}$ to make things easier.}

Below, we list the rules of the labelled relational sequent calculus \GTK for the basic normal modal logic K, where cut, weakening, contraction, and necessitation are admissible rules (see for instance \cite{negri2005proof}). In the list below, we explicitly mention the cut rule and we do not include the rules for negation. The propositional and modal rules are all invertible. 

\begin{center}
\begin{tabular}{rl}

\mc{2}{c}{\rule[-1.85mm]{0mm}{8mm}\textbf{Initial rules and cut rule}\rule[-1.85mm]{0mm}{6mm}} \\

\mc{2}{c}{
\AXC{$ \ $}
\LL{$\abot_L$}
\UIC{$\Gamma, \nx : \abot \vd \Delta$}
\DP
 \ \   
\AXC{$\ $}
\RL{Id$_{x:p}$}
\UIC{$\Gamma, \nx : p \vd \nx : p, \Delta$}
\DP
}
\end{tabular}

\begin{tabular}{c}
\\
\AXC{$\Gamma \vd x: p, \Delta$}
\AXC{$\Gamma', x: p \vd \Delta'$}
\RL{Cut}
\BIC{$\Gamma, \Gamma' \vd \Delta, \Delta'$}
\DP
\end{tabular}

\begin{tabular}{rl}

\mc{2}{c}{\rule[-1.85mm]{0mm}{8mm} \textbf{Invertible propositional rules}\rule[-1.85mm]{0mm}{8mm}} \\

\AXC{$\Gamma, \nx : A, \nx : B \vd \Delta$}
\LL{$\aand_L$}
\UIC{$\Gamma, \nx : A \aand B \vd \Delta$}
\DP
 & 
\AXC{$\Gamma \vd \nx : A, \Delta$}
\AXC{$\Gamma \vd \nx : B, \Delta$}
\RL{$\aand_R$}
\BIC{$\Gamma \vd \nx : A \aand B, \Delta$}
\DP
 \\

 & \\

\AXC{$\Gamma, \nx : A \vd \Delta$}
\AXC{$\Gamma, \nx : B \vd \Delta$}
\LL{$\aor_L$}
\BIC{$\Gamma, \nx : A \aor B \vd \Delta$}
\DP
 & 
\AXC{$\Gamma \vd \nx : A, \nx : B, \Delta$}
\RL{$\aor_R$}
\UIC{$\Gamma \vd \nx : A \aor B, \Delta$}
\DP
 \\

 & \\

\AXC{$\Gamma \vd \nx : A, \Delta$}
\AXC{$\Gamma, \nx : B \vd \Delta$}
\LL{$\ararr_L$}
\BIC{$\Gamma, \nx : A \ararr B \vd \Delta$}
\DP
 & 
\AXC{$\Gamma, \nx : A \vd \nx : B, \Delta$}
\RL{$\ararr_R$}
\UIC{$\Gamma \vd \nx : A \ararr B, \Delta$}
\DP
 \\
 
 & \\

%\AXC{$\Gamma \vd \nx : A, \Delta$}
%\LL{$\neg_L$}
%\UIC{$\Gamma, \nx : \neg A \vd \Delta$}
%\DP

% & 
%\AXC{$\Gamma, \nx : A \vd \Delta$}
%\RL{$\neg_R$}
%\UIC{$\Gamma \vd \nx : \neg A, \Delta$}
%\DP
% \\

\end{tabular}

\begin{tabular}{rl}

\mc{2}{c}{\rule[-1.85mm]{0mm}{8mm}\textbf{Invertible modal rules${}^\ast$}\rule[-1.85mm]{0mm}{8mm}} \\

\AXC{$\nx\NR\ny, \Gamma, \nx : \wbox A, \ny : A \vd \Delta$}
\LL{$\wbox_L$}
\UIC{$\nx\NR\ny, \Gamma, \nx : \wbox A \vd \Delta$}
\DP
 & 
\AXC{$\nx\NR\ny, \Gamma \vd \ny : A, \Delta$}
\RL{$\wbox_R$}
\UIC{$\Gamma \vd \nx : \wbox A, \Delta$}
\DP
 \\

 & \\

\AXC{$\nx\NR\ny, \Gamma, \ny : A \vd \Delta$}
\LL{$\wdia_L$}
\UIC{$\Gamma, \nx : \wdia A \vd \Delta$}
\DP
 & 
\AXC{$\nx\NR\ny, \Gamma \vd \ny : A, \nx : \wdia A, \Delta$}
\RL{$\wdia_R$}
\UIC{$\nx\NR\ny, \Gamma \vd \nx : \wdia A, \Delta$}
\DP

 \\
 
\end{tabular}

\begin{tabular}{rl}

\mc{2}{c}{\rule[-1.85mm]{0mm}{8mm}\textbf{Equality rules}\rule[-1.85mm]{0mm}{8mm}} \\

\AXC{$x=x, \Gamma \vd \Delta$}
\LL{\small Eq-Ref}
\UIC{$\Gamma \vd \Delta$}
\DP
 & 
\AXC{$y=z, x=y, x=z, \Gamma \vd \Delta$}
\LL{\small Eq-Trans}
\UIC{$x=y, x=z, \Gamma \vd \Delta$}
\DP
 \\

 & \\

\AXC{$yRz, x=y, xRz, \Gamma \vd \Delta$}
\LL{\small $\textrm{Repl}_{R1}$}
\UIC{$x=y, xRz, \Gamma \vd \Delta$}
\DP
 & 
\AXC{$xRz, y=z, xRy, \Gamma \vd \Delta$}
\LL{\small $\textrm{Repl}_{R2}$}
\UIC{$y=z, xRy, \Gamma \vd \Delta$}
\DP

 \\

 & \\

\mc{2}{c}{
\AXC{$x=y, y : A, x : A, \Gamma \vd \Delta$}
\LL{\small Repl}
\UIC{$x=y, x : A, \Gamma \vd \Delta$}
\DP}\\

 & 
 
\end{tabular}

\end{center}

\noindent ${}^\ast$Side condition: the label $y$ must not occur in the conclusion of $\wbox_R$ and $\wdia_L$.

\begin{remark}
The logical rules above (namely Propositional and Modal rules) reflect the semantic clauses of each connective in the intended Kripke semantics. %, where (i) the accessibility relation $R$ plays a role only in the case of modalities, (ii) the order-theoretic properties of the connectives are codified in rule form.
Logical rules can be  grouped together as \emph{tonicity rules} ($\aand_R, \aor_L, \ararr_L, \wbox_L, \wdia_R$) versus \emph{translation rules} ($\aand_L, \aor_R, \ararr_R, \wbox_R, \wdia_L$). Tonicity rules specify the arity of a connective (i.e.~a connective of arity $n$ is introduced by a tonicity rule with $n$ premises) and its tonicity (i.e.~if the connective is positive or negative in each coordinate). The translation rules convert a proxy occurring in the premise (either the comma or a relational atom) into a logical connective (namely, the main connective of the principal formula occurring in the conclusion).%\footnote{For instance, $\aand_R$ codifies \red{????? ... is sound iff?} that $\aand$ is binary, monotone in both coordinates, and is the right-adjoint of the diagonal map, while $\aand_L$ implies that commas in precedent position are translated into $\aand$ at the logical level. Similarly, $\wbox_L$ codifies that $\wbox$ is unary, monotone, and  is a normal box-operator, while $\wbox_R$ and $\wdia_L$ implies that $\wbox$ and $\wdia$ are interpreted on the very same relation etc.}
\end{remark}

Below we list the non-invertible versions of the tonicity logical rules. We sometimes refer to them as multiplicative rules.

\begin{center}
\begin{tabular}{rl}

\mc{2}{c}{\rule[-1.85mm]{0mm}{8mm} \textbf{Non-invertible tonicity propositional rules}\rule[-1.85mm]{0mm}{8mm}} \\

\AXC{$\Gamma, \nx : A \vd \Delta$}
\AXC{$\Gamma', \nx : B \vd \Delta'$}
\LL{$\aor_L$}
\BIC{$\Gamma, \Gamma', \nx : A \aor B \vd \Delta, \Delta'$}
\DP

 & 
 
\AXC{$\Gamma \vd \nx : A, \Delta$}
\AXC{$\Gamma' \vd \nx : B, \Delta'$}
\RL{$\aand_R$}
\BIC{$\Gamma, \Gamma', \vd \nx : A \aand B, \Delta, \Delta'$}
\DP

 \\ 

 & \\

\mc{2}{c}{ 
\AXC{$\Gamma \vd \nx : A, \Delta$}
\AXC{$\Gamma', \nx : B \vd \Delta'$}
\LL{$\ararr_L$}
\BIC{$\Gamma, \Gamma', \nx : A \ararr B \vd \Delta, \Delta'$}
\DP}
 \\
 
%\AXC{$\Gamma \vd \nx : A, \Delta$}
%\LL{$\neg_L$}
%\UIC{$\Gamma, \nx : \neg A \vd \Delta$}
%\DP

% & 
%\AXC{$\Gamma, \nx : A \vd \Delta$}
%\RL{$\neg_R$}
%\UIC{$\Gamma \vd \nx : \neg A, \Delta$}
%\DP
% \\
\end{tabular}

\begin{tabular}{rl}
\mc{2}{c}{\rule[-1.85mm]{0mm}{8mm}\textbf{Non-invertible tonicity modal rules}\rule[-1.85mm]{0mm}{8mm}} \\

\AXC{$\nx\NR\ny, \Gamma, \ny : A \vd \Delta$}
\LL{$\wbox_L$}
\UIC{$\nx\NR\ny, \Gamma, \nx : \wbox A \vd \Delta$}
\DP
 & 
\AXC{$\nx\NR\ny, \Gamma \vd \ny : A, \Delta$}
\RL{$\wdia_R$}
\UIC{$\nx\NR\ny, \Gamma \vd \nx : \wdia A, \Delta$}
\DP
 \\
\end{tabular}

\end{center}

%\begin{remark}
%\label{AtomicLeaf}
%For any zeroary rule Id$_{x:p}$, we say that the occurrences $x: p$ and $x: p$ in precedent and, respectively, succedent position are in \emph{axiom linking}. We say that Id$_{x:p}$ is an \emph{atomic identity}, if the formulas in axiom linking are both atomic formulas $x: p$ (where it could be the case that $p = \abot$ or $p = \aatop$), $\Gamma$ is empty or contains relational atoms or complex labelled formula of the form $\wbox A$, and $\Delta$ contains at most complex formula of the form $\wdia A$. We say that $\abot_L$ is an \emph{atomic initial rule}, if $\Gamma$ is empty, and $\Delta$ contains one atomic formula $x:p$ and at most one complex formula $x: \wdia p$. An \emph{atomic leaf} is either an atomic identity or an atomic initial rule.\footnote{Notice that, making use of explicit weakening rules, every proof can be transformed into a proof where all the leaves are either atomic identities or atomic initial rules. \textcolor{red}{GG: Is it the case? To be confirmed.}}
%\end{remark}

\begin{lemma}
\label{lemma: phi implies phi}
For any modal formula $A$, the sequent \,$\Gamma, x: A\vd x:A, \Delta$\, is derivable in G3K. 
\end{lemma}
\begin{proof}
By induction on $A$. The cases of $A\coloneqq \bot$ and $A\coloneqq p\in \mathsf{Prop}$ are immediate. If $A: =\ast(\ol{A'})$ where $\ast\in \{\wbox, \ararr, \vee\}$, then the required proof is obtained by applying, from bottom to top, $\ast_R$ to the occurrence of $A$ in succedent position,  followed by a bottom-up application of $\ast_L$ to the occurrence of $A$ in precedent position, and then using the induction hypothesis on each $A'$ in $\ol{A'}$. Similarly, the required proof if $A: =\ast(\ol{A'})$ where $\ast\in \{\wdia,  \aand\}$,  is obtained by applying, from bottom to top, $\ast_L$ followed by $\ast_R$.
\end{proof}
Notice that the derivation generated in the proof of the lemma above introduces {\em every} subformula of each occurrence of $\varphi$   via a logical rule, and, modulo renaming variables, we can assume w.l.o.g.~that every new label introduced proceeding bottom-up be fresh in the entire derivation (and not just in every branch, as already required by the side conditions of the rule $\wbox_R$ and $\wdia_L$).
Below we recall the definition of a geometric implication. 
\begin{definition}(cf.~\cite[Section 3]{Neg03})
\label{def:GeometricImplication}
A \emph{geometric implication} is a first-order sentence of the form 
$$\forall\ol{x}(s \ararr t),$$
where both $s$ and $t$ are {\em geometric formulas}, i.e.~first-order formulas not containing $\ararr$ or $\forall$. Geometric implications can be equivalently rewritten as conjunctions of \emph{geometric axioms}, namely, sentences of the type
\begin{center}
$\forall \ol{x}(P_1 \wedge ... \wedge P_m \ararr \ol{\exists y_1}M_1 \vee ... \vee \ol{\exists y_n}M_n)$
\end{center}
where each $P_i$ is an atomic formula with no free occurrences of any variable $y$ in $\ol{y}$, and $M_j$ is a conjunction of atomic formulas $Q_{j_1} \wedge ... \wedge Q_{j_{k_j}}$. The rule scheme corresponding to geometric axioms takes the form
\begin{center}
\AXC{$\ol{Q_1}[\ol{y_1}/\ol{z_1}], \ol{P}, \Gamma \vd \Delta$}
\AXC{...}
\AXC{$\ol{Q_n}[\ol{y_n}/\ol{z_n}], \ol{P}, \Gamma \vd \Delta$}
\RL{$GR$}
\TIC{$\ol{P}, \Gamma \vd \Delta$}
\DP
\end{center}
where $\ol{Q_i}[\ol{y_i}/\ol{z_i}]$ denotes the simultaneous replacement of each $z$ in $\ol{z_i}$ with the corresponding $y$ in $\ol{y_i}$, in every $Q$ in $\ol{Q_i}$. In this scheme, %has the condition that 
the eigenvariables in $\ol{y_i}$ are not free in $\ol{P}, \Delta, \Gamma$.  Rules corresponding to geometric axioms are referred to as \emph{geometric (labelled) rules}.
\end{definition}

A \emph{geometric labelled calculus} is any extension of \GTK with geometric labelled rules. 
\begin{theorem}(cf.~\cite[Theorem 4.13]{negri2005proof})
Any geometric labelled calculus preserves cut admissibility.
\end{theorem}

%% file: paper_MASSA/AnalyticInductive.tex
\subsection{Analytic inductive formulas}
\label{sec:AnalyticInductiveFormulas}

In this subsection, we specialize and adapt the definition of analytic inductive inequality (cf.~\cite[Definition 55]{GMPTZ}, \cite[Definition 2.14]{de2021slanted}, \cite[Section 2.3]{ChnGrePalTzi21}) to the language and properties of classical modal logic.

The language of the basic normal modal logic K  is recursively defined from a set $\mathsf{Prop}$ of proposition variables as follows:
\begin{center}
$\varphi ::= p \ | \ \bot \ | \ \neg \varphi \ |\  \varphi \wedge \varphi \ |\ \varphi \vee \varphi \ |\ \varphi \rightarrow \varphi \ |\  \Diamond \varphi \ |\ \Box \varphi$,
\end{center}
where $p$ ranges over $\mathsf{Prop}$. 
In what follows, we will need to keep track of the multiplicity of occurrences of  proposition variables in formulas, as well as the order-theoretic properties of the various coordinates of the term-functions associated with formulas. Therefore, we will write e.g.~$\psi(!\ol{x})$ to signify that each variable in the vector $\ol{x}$ of  placeholder variables   occurs exactly once in $\psi$.     Moreover, we will write e.g.~$\psi(!\ol{x}, !\ol{y})$ to mean that $\psi$ (resp.~the term-function $\psi^\mathbb{A}$ in a modal algebra $\mathbb{A}$) is positive (resp.~monotone) in each $x$-coordinate and negative (resp.~antitone) in each $y$-coordinate. In other contexts, we will sometimes need to group coordinates according to different criteria. In each context in which  this is the case, we will specifically indicate these criteria. Negative (resp.~positive) {\em Skeleton} formulas $\psi(!\ol{x}, !\ol{y})$ (resp.~$\varphi(!\ol{x}, !\ol{y})$) are defined by simultaneous recursion as follows:
\begin{center}
\begin{tabular}{rc l}
$\psi (!\ol{x}, !\ol{y})$ &$ ::=$ &$ x \ | \ \neg \varphi \ |\  \psi \wedge \psi \ |\   \psi \vee \psi \ |\ \varphi \rightarrow \psi  \ |\ \Box \psi$,\\

$\varphi(!\ol{x}, !\ol{y})$ &$ ::=$ &$  x \ | \ \neg \psi \ |\  \varphi \wedge \varphi \ |\  \varphi \vee \varphi \ |\  \Diamond \varphi$.
\end{tabular}
\end{center}
Positive Skeleton formulas will sometimes be referred to as {\em negative PIA} formulas. {\em Definite} negative Skeleton (resp.~PIA) formulas are  defined by simultaneous recursion as follows:
\begin{center}
\begin{tabular}{rc l}
$\psi (!\ol{x}, !\ol{y})$ &$ ::=$ &$ x \ | \ \neg \varphi \ |\    \psi \vee \psi \ |\ \varphi \rightarrow \psi  \ |\ \Box \psi$,\\

$\varphi(!\ol{x}, !\ol{y})$ &$ ::=$ &$  x \ | \ \neg \psi \ |\  \varphi \wedge \varphi  \ |\  \Diamond \varphi$.
\end{tabular}
\end{center}
Modulo exhaustively distributing all the other connectives over  $\aor$ and $\aand$, any negative Skeleton (resp.~PIA) formula can be equivalently rewritten as a conjunction (resp.~disjunction) of definite negative Skeleton (resp.~PIA) formulas (cf.~\cite[Lemma 2.9]{ChnGrePalTzi21}). 
\begin{definition}
A modal formula $\psi'(\ol{p})$ is (negative) {\em analytic inductive}  if its negative normal form (NNF) is $ \psi(\ol{\beta}/!\ol{x}, \ol{\delta}/!\ol{y})$ such that:
\begin{enumerate}
    \item $\psi(!\ol{x}, !\ol{y})$ (which we refer to as the {\em Skeleton} of $\psi'$) is a negative Skeleton formula, and is monotone {\em both} in its $x$-coordinates and in its $y$-coordinates; \item each $\beta$ in $\ol{\beta}$ and  $\delta$ in $ \ol{\delta}$ is a negative PIA formula;
    \item the term-function $\delta^\mathbb{A}(!\ol{x})$ associated with each $\delta(\ol{p}/!\ol{x})$ in $\ol{\delta}$ is monotone in each coordinate;
    \item the term-function $\beta^\mathbb{A}(!\ol{x}, !\ol{y})$ associated with each $\beta$ in $\ol{\beta}$ is monotone in each $x$-coordinate and antitone in each $y$-coordinate; 
    \item the transitive closure $<_\Omega$ of the relation $\Omega$ (defined below) is a well-founded strict order  on $\ol{p}$, where for all $p, p'$ in  $\ol{p}$,  $(p, p')\in \Omega$ iff some $\beta\in \ol{\beta}$ exists s.t.~$\beta = \beta(\ol{p_1}/!\ol{x}, \ol{p_2}/!\ol{y})$ and $p'$ occurs in $\ol{p_1}$ and $p$ occurs in $\ol{p_2}$, and the lowest common node in the branches ending in $p'$ and $p$ in the generation tree of $\beta$ is a $\aand$-node. 
\end{enumerate}
In an analytic inductive formula $\psi'$ as above, the variable occurrences in the $y$-coordinates of each $\beta$ in $\ol{\beta}$ are referred to as the {\em critical} occurrences in $\psi'$. All the other variable occurrences are {\em non-critical}.
An analytic inductive formula is {\em Sahlqvist} if the relation $\Omega$ is empty, and is {\em definite} if its Skeleton is definite.
\end{definition}
As discussed above, for any analytic inductive formula $\psi': = \psi(\ol{\beta}/!\ol{x}, \ol{\delta}/!\ol{y})$, any negative PIA subformula $\beta$ and $\delta$ of $\psi'$  can be equivalently rewritten as a disjunction of definite negative PIA formulas (cf.~\cite[Lemma 2.9]{ChnGrePalTzi21}). Hence,  once these $\vee$-nodes have reached the root of $\beta$ by distributing all the other connectives over them, they can all be considered part of the Skeleton of $\psi'$. Hence, when representing an analytic inductive formula $\psi'$ as $\psi(\ol{\beta}/!\ol{x}, \ol{\delta}/!\ol{y})$, we can assume w.l.o.g.~that each $\beta$ and $\delta$ is a  {\em definite} negative PIA formula, and that there is {\em exactly one} critical occurrence of a proposition variable in each $\beta$ in $\ol{\beta}$. To emphasise this, we  sometimes write $\beta$ as $\beta_p$. 

%Below we provide a few examples. 

\begin{example}
\label{examples}
\begin{enumerate}
\item The formula $\psi'(p): = \wdia p\ararr \wbox p$  can be rewritten in NNF as $\psi(\beta/x, \delta/y)$ where $\psi(x, y): = \wbox x\aor \wbox y$, and $\beta(p): =  \neg p$, and $\delta(p): = p$, and is hence (negative) analytic Sahlqvist.
\item The formula $\psi'(p): = \wbox p\ararr \wdia p$  can be rewritten in NNF as $\psi(\beta/x, \delta/y)$ where $\psi(x, y): =  x\aor  y$, and $\beta(p): = \wdia  \neg p$ and $ \delta(p): = \wdia p$, and is hence (negative) analytic Sahlqvist.
\item The formula $\psi'(p): = \wdia \wbox p\ararr \wbox \wdia p$  can be rewritten in NNF as $\psi(\beta/x, \delta/y)$ where $\psi(x, y): = \wbox x\aor \wbox y$, and $\beta(p): = \wdia \neg p$ and $ \delta(p): = \wdia p$, and is hence (negative) analytic Sahlqvist.
\item The formula $\psi'(p_1, p_2): = \wbox (p_1 \ararr p_2) \ararr (\wbox p_1 \ararr \wbox p_2)$ can be  rewritten in NNF as %$\wdia_y (p_1 \aand \neg p_2) \aor (\wdia_z \neg p_1 \aor \wbox_t p_2)$ and hence represented as
$\psi(\beta_1/x_1, \beta_2/x_2, \delta/y)$ where $\psi(x_1, x_2, y): = x_1\aor (x_2 \aor \wbox_t y)$, and $\beta_1(p_1, p_2): = \wdia_y (p_1 \aand \neg p_2)$ and $\beta_2 (p_1): = \wdia \neg p_1$ and $ \delta(p_2): =  p_2$, and is hence (negative) analytic inductive with $p_1<_{\Omega} p_2$.
\item The formula $\psi'(p_1, p_2): = \wbox (\wbox p_1 \ararr p_2) \lor \wbox (\wbox p_2 \ararr p_1)$ can be  rewritten in NNF as %$\wbox (\wdia \neg p_1 \aor p_2) \lor \wbox (\wdia \neg p_2 \aor p_1)$ and hence represented as
$\psi(\beta_1/x_1, \beta_2/x_2, \delta_1/y_1, \delta_2/y_2)$ where $\psi(x_1, x_2, y_1, y_2): = \wbox (x_1 \aor y_2) \lor \wbox (x_2 \aor y_1)$, and $\beta_1(p_1): = \wdia \neg p_1$ and $\beta_2 (p_2): = \wdia \neg p_2$, and $ \delta_1(p_1): =  p_1$ and $ \delta_(p_2): =  p_2$, and is hence (negative) analytic Sahlqvist.
\end{enumerate}
\end{example}

\begin{theorem}
\label{thm: inductive have fo corr}
(cf.~\cite[Theorem 37]{Goranko:Vakarelov:2006}) Every (analytic) inductive formula has a first-order correspondent.
\end{theorem}

\subsection{Analytic inductive formulas with signed trees}\label{Inductive:Fmls:Section}
In this subsection, we recall the definitions of inductive LE-inequalities introduced in \cite{conradie2019algorithmic} and their corresponding `analytic' restrictions introduced in \cite{GMPTZ} in the boolean setting. This subsection generalizes some of the definitions already given in Subsection \ref{sec:AnalyticInductiveFormulas}, and it will be important when proving the termination and soundness of \Algname.

\begin{definition}[Signed generation tree]
	\label{def: signed gen tree}
	The \emph{positive} (resp.~\emph{negative}) {\em generation tree} of any formula $s$ is defined by labelling the root node of the generation tree (i.e.~syntax tree) of $s$ with the sign $+$ (resp.~$-$), and then propagating the labelling on each remaining node as follows:
	\begin{itemize}
		%\item The root node $+s$ (resp.~$-s$) is the root node of the positive (resp.~negative) generation tree of $s$ signed with + (resp.~$-$).
		\item For any node labelled with $ \lor$, $\land$, $\wdia$, or $\wbox$ ,assign the same sign to its children nodes.
		%\item If a node is labelled with $\lhd$, $\rhd$, assign the opposite sign to its child node.
		\item For any node labelled with $\neg$, assign the opposite sign to its children; for any node labelled with $\rightarrow$, assign the opposite sign to its left children and the same sign to its right children.
	\end{itemize}
	Nodes in signed generation trees are \emph{positive} (resp.~\emph{negative}) if they are signed $+$ (resp.~$-$).
\end{definition}

Signed generation trees will  mostly be used in the context of term inequalities $s\leq t$. In this context, we will typically consider the positive generation tree $+s$ for the left-hand side and the negative one $-t$ for the right-hand side. We will also say that a term-inequality $s\leq t$ is \emph{uniform} in a given variable $p$ if all occurrences of $p$ in both $+s$ and $-t$ have the same sign, and that $s\leq t$ is $\varepsilon$-\emph{uniform} in a (sub)array $\overline{r}$ of its variables if each $r\in \overline{r}$ such that $\varepsilon(r) = 1$ (resp.~$\varepsilon(r) = \partial$) occurs positively (resp.~negatively) in $s\leq t$. 

For any term $s(p_1,\ldots p_n)$, any order-type $\varepsilon$ over $n$, and any $1 \leq i \leq n$, an \emph{$\varepsilon$-critical node} in a signed generation tree of $s$ is a leaf node $+p_i$ if $\varepsilon(i) = 1$, and a leaf node $-p_i$ if $\varepsilon (i) = \partial$. An $\varepsilon$-{\em critical branch} in the tree is a branch the leaf of which is an $\varepsilon$-critical node. Variable occurrences corresponding to $\varepsilon$-critical nodes are those used in the runs of the various versions of the algorithm ALBA (cf.~\cite{CoGhPa13, conradie2019algorithmic}) to compute the minimal valuations. 
For every term $s(p_1,\ldots p_n)$ and every order-type $\varepsilon$, we say that $+s$ (resp.~$-s$) {\em agrees with} $\varepsilon$, and write $\varepsilon(+s)$ (resp.~$\varepsilon(-s)$), if every leaf in the signed generation tree of $+s$ (resp.~$-s$) is $\varepsilon$-critical.
%In other words, $\varepsilon(+s)$ (resp.~$\varepsilon(-s)$) means that all variable occurrences corresponding to leaves of $+s$ (resp.~$-s$) are  $\varepsilon$-critical.
 We will also write $+s'\prec \ast s$ (resp.~$-s'\prec \ast s$) to indicate that the subterm $s'$ inherits the positive (resp.~negative) sign from the signed generation tree $\ast s$. Finally, we will write $\varepsilon(\gamma) \prec \ast s$ (resp.~$\varepsilon^\partial(\gamma) \prec \ast s$) to indicate that the signed subtree $\gamma$, with the sign inherited from $\ast s$, agrees with $\varepsilon$ (resp.~with $\varepsilon^\partial$).

\begin{definition}
	\label{def:good:branch}
	A branch in a signed generation tree $\ast s$, with $\ast \in \{+, - \}$, is called a \emph{good branch} if it is the concatenation of two paths $P_1$ and $P_2$, one of which may possibly be of length $0$, such that $P_1$ is a path from the leaf consisting (apart from variable nodes) only of PIA-nodes, and $P_2$ consists (apart from variable nodes) only of Skeleton-nodes. 
	A good branch is \emph{Skeleton} if the length of $P_1$ is $0$.
\end{definition}

\begin{definition}[Inductive inequalities]\label{Inducive:Ineq:Def}
	For any order-type $\varepsilon$ and any irreflexive and transitive relation (i.e.~strict partial order) $\Omega$ on $p_1,\ldots p_n$, the signed generation tree $*s$ $(* \in \{-, + \})$ of a term $s(p_1,\ldots p_n)$ is \emph{$(\Omega, \varepsilon)$-inductive} if for all $1 \leq i \leq n$
	\begin{enumerate}
		\item every $\varepsilon$-critical branch with leaf $p_i$ is good (cf.~Definition \ref{def:good:branch});
		\item for every maximal PIA term\footnote{A maximal PIA term of $*s$ is a subformula $\alpha$ of $*s$ (with the inherited sign) such that all its connectives (if any) are PIA and every subformula properly between $\alpha$ and $*s$ contains a Skeleton connective.} $\alpha(q_1,\ldots,q_m)$ containing a critical propositional atom $q_j$
		\begin{enumerate}
			\item $q_i$ is not critical for every $i\neq j$;
			\item $p_i <_{\Omega} p_j$ for every $i\neq j$.
		\end{enumerate}
	\end{enumerate}
	
	We will refer to $<_{\Omega}$ as the \emph{dependency order} on the variables. An inequality $s \leq t$ is \emph{$(\Omega, \varepsilon)$-inductive} if the signed generation trees $+s$ and $-t$ are $(\Omega, \varepsilon)$-inductive. An inequality $s \leq t$ is \emph{inductive} if it is $(\Omega, \varepsilon)$-inductive for some $\Omega$ and $\varepsilon$.
\end{definition}

In what follows, we refer to formulas $\varphi$ such that only PIA nodes occur in $+\varphi$ (resp.~$-\varphi$) as {\em positive} (resp.~{\em negative}) {\em PIA formulas}, and to formulas $\xi$ such that only Skeleton nodes occur in $+\xi$ (resp.~$-\xi$) as {\em positive} (resp.~{\em negative}) {\em Skeleton formulas}\label{page: positive negative PIA}.
Skeleton formulas $\xi$  in which no $\Delta$-adjoint nodes (i.e.~$-\wedge$ and $+\vee$) occur in $\ast \xi$
are referred to as {\em definite}. The specific order-theoretic properties of definite Skeleton and PIA  formulas entail that these are exactly the formulas which can be fully captured at the structural level in display calculi (cf.~\cite{GMPTZ}). 

\begin{lemma} 
\label{lemma: reduction to definite} For every LE-language $\mathcal{L}$,
\begin{enumerate}
\item if $\gamma$ is a positive PIA (i.e.~negative Skeleton) $\mathcal{L}$-formula,  then $\gamma$ is equivalent to $\bigwedge_{i\in I}\gamma_i$ for some finite set of definite positive PIA (i.e.~negative Skeleton) formulas $\gamma_i$;
\item if $\delta$ is a negative PIA (i.e.~positive Skeleton) $\mathcal{L}$-formula,  then $\delta$ is equivalent to $\bigvee_{j\in j}\delta_j$ for some finite set of definite negative PIA (i.e.~positive Skeleton) formulas $\delta_j$.
\end{enumerate}
\end{lemma}

\begin{definition}[Analytic inductive inequalities]
	\label{def:type5}
	For every order-type $\varepsilon$ and every irreflexive and transitive relation $\Omega$ on the variables $p_1,\ldots p_n$,
	the signed generation tree $\ast s$ ($\ast\in \{+, -\}$) of a term $s(p_1,\ldots p_n)$ is \emph{analytic $(\Omega, \varepsilon)$-inductive}  if
	
	\begin{enumerate}%\marginnote{We still need to sort out the occurences of top and bottom}
		\item $\ast s$ is $(\Omega, \varepsilon)$-inductive (cf.~Definition \ref{Inducive:Ineq:Def});
		\item every branch of $\ast s$ is good (cf.~Definition \ref{def:good:branch}).
	\end{enumerate}	
	
	An inequality $s \leq t$ is \emph{analytic $(\Omega, \varepsilon)$-inductive}   if $+s$ and $-t$ are both analytic  $(\Omega, \varepsilon)$-inductive. An inequality $s \leq t$ is \emph{analytic inductive} if is analytic $(\Omega, \varepsilon)$-inductive  for some $\Omega$ and $\varepsilon$. An analytic inductive inequality is {\em definite} if no $-\wedge$ and $+\vee$ nodes occur in its Skeleton.
	\end{definition}	
	
Figure \ref{AnalyticInductiveInequalities} provides a visual representation of the shape of analytic inductive inequalities, where all branches (even non-critical) have to be good.

\begin{figure}[h]
\centering
	\begin{tikzpicture}
	\draw (-5,-1.5) -- (-3,1.5) node[above]{\Large$+$} ;
	\draw (-5,-1.5) -- (-1,-1.5) ;
	\draw (-3,1.5) -- (-1,-1.5);
	\draw (-5.5,0) node{Skeleton ($P2$)} ;
	\draw[dashed] (-3,1.5) -- (-4,-1.5);
	\draw[dashed] (-3,1.5) -- (-2,-1.5);
	\draw (-4,-1.5) --(-4.8,-3);
	\draw (-4.8,-3) --(-3.2,-3);
	\draw (-3.2,-3) --(-4,-1.5);
	\draw[dashed] (-4,-1.5) -- (-4,-3);
	\draw[fill] (-4,-3) circle[radius=.1] node[below]{$+p$};
	\draw
	(-2,-1.5) -- (-2.8,-3) -- (-1.2,-3) -- (-2,-1.5);
	%\fill[pattern=north east lines](-2,-1.5) -- (-2.8,-3) -- (-1.2,-3);
	\draw (-2,-3.25)node{$\gamma$};
	\draw (-5.5,-2.25) node{PIA ($P1$)} ;
	\draw (0,0) node{$\leq$};
	\draw (5,-1.5) -- (3,1.5) node[above]{\Large$-$} ;
	\draw (5,-1.5) -- (1,-1.5) ;
	\draw (3,1.5) -- (1,-1.5);
	\draw (5.5,0) node{Skeleton ($P2$)} ;
	\draw[dashed] (3,1.5) -- (4,-1.5);
	\draw[dashed] (3,1.5) -- (2,-1.5);
	\draw (2,-1.5) --(2.8,-3);
	\draw (2.8,-3) --(1.2,-3);
	\draw (1.2,-3) --(2,-1.5);
	\draw[dashed] (2,-1.5) -- (2,-3);
	\draw[fill] (2,-3) circle[radius=.1] node[below]{$+p$};
	\draw
	(4,-1.5) -- (4.8,-3) -- (3.2,-3) -- (4, -1.5);
	%\fill[pattern=north east lines](4,-1.5) -- (4.8,-3) -- (3.2,-3) -- (4, -1.5);
	\draw (4,-3.25)node{$\gamma'$};
	\draw (0.5,-2.25) node{PIA ($P1$)} ;
	\draw (-1,-2.25) node{PIA} ;
         \draw (5,-2.25) node{PIA} ;
	\end{tikzpicture}
\caption{The shape of analytic inductive inequalities} \label{AnalyticInductiveInequalities}
\end{figure}

\begin{notation}\label{notation: representations of signed generation trees}
We adopt the convention that in graphical representations of signed generation trees the squared variable occurrences are the $\varepsilon$-critical ones, the doubly circled nodes are the Skeleton ones and the single-circle ones are PIA nodes.
\end{notation}

	\begin{example}\label{Ex:Church-Rosser et al}
	\label{example:inductive and analytic inductive}
	The inequality $p\le \Diamond\Box p$ is $\varepsilon$-inductive %\textcolor{red}{(and in particular also Sahlqvist \cite{Sahlqvist})} 
for $\varepsilon(p) = 1$, but is not analytic inductive for any order-type, because the negative generation tree of $\Diamond \Box p$, which has only one branch, is not good. The Church-Rosser inequality $\Diamond \Box p \le \Box \Diamond p$ is analytic $\varepsilon$-inductive for every order-type. 

The inequality $p \ararr (q \ararr r) \le ((p\ararr q) \ararr (\wbox p \ararr r))\wedge \Diamond r$  is %not Sahlqvist for any order-type: indeed, both the positive and the negative occurrence of $q$ occur under the scope of an SRR-connective. However, it 
is an analytic $(\Omega, \varepsilon)$-inductive inequality, e.g.~for $p <_\Omega q  <_\Omega r $ and $\varepsilon(p,q,r)=(1,1,\partial)$. 

Below, we represent the signed generation trees pertaining to the inequalities above (see Notation \ref{notation: representations of signed generation trees}): %\marginnote{Jinsheng, I think that in order for all the pictures to be coherent, in the signed generation tree of $p\leq \Diamond \Box p$ the occurrrence of $+p$ needs to be into a square, not a circle. G: I fixed this. Moreover, I added a curly bracket pointing to $\alpha_p$?}
	\begin{center}
		\begin{tikzpicture}
%%%first inequality
		\node at(-3,0){
			\begin{tikzpicture}
			\tikzstyle{level 1}=[level distance=1cm, sibling distance=2.5cm]
			\tikzstyle{level 2}=[level distance=1cm, sibling distance=1.5cm]
			\tikzstyle{level 3}=[level distance=1 cm, sibling distance=1.5cm]
			\node[draw] at (-1,0) {$\begin{aligned} +p \end{aligned}$}
			;
			\node at (0,0) {$\le$}; 
			
			\node[PIA] at (1,0) {$\begin{aligned} -\wdia \end{aligned}$}
			child{node[Ske]{$\begin{aligned} -\wbox \end{aligned}$}
				child{node{$-p$}}
			}
			;
			\end{tikzpicture}
		};
%%%Curch-Rosser inequality
		\node at(4,0){
			\begin{tikzpicture}
			\tikzstyle{level 1}=[level distance=1cm, sibling distance=2.5cm]
			\tikzstyle{level 2}=[level distance=1cm, sibling distance=1.5cm]
			\tikzstyle{level 3}=[level distance=1 cm, sibling distance=1.5cm]
			\node[Ske] at (-1.5,0) {$\begin{aligned} +\wdia \end{aligned}$}
			child{node[PIA]{$\begin{aligned} +\wbox \end{aligned}$}
				child{node[draw]{$+p$}}
			}
			;
			\node at (0,0) {$\le$}; 
			
			\node[Ske] at (1.5,0) {$\begin{aligned} -\wbox \end{aligned}$}
			child{node[PIA]{$\begin{aligned} -\wdia \end{aligned}$}
				child{node{$-p$}}
			}
			;
			\node[rotate = -90] at (1, -1.5) {$\underbrace{\hspace{1.3cm}}$};
			\node at (0.75,-1.5) {$\rdelta$};
			% \draw[help lines] (-4,-4) grid (4,4);
			\node[rotate = -90] at (-2, -1.5) {$\underbrace{\hspace{1.3cm}}$};
			\node at (-2.33,-1.5) {$\textcolor{blue}{\alpha_p}$};
			% \draw[help lines] (-4,-4) grid (4,4);
			\end{tikzpicture}
		};
		%\caption{Signed generation tree for $\Box (p \rightarrow  q) \to \Box (\Box p\to\Box q)$}
		%\label{fig:fisher-servi}
		\node at (0.7,-4){
			\begin{tikzpicture}
			\tikzstyle{level 1}=[level distance=1cm, sibling distance=2.5cm]
			\tikzstyle{level 2}=[level distance=1cm, sibling distance=2.5cm]
			\tikzstyle{level 3}=[level distance=1 cm, sibling distance=1.5cm]
			\node[PIA] at (-3.5,0) {$\begin{aligned} +\ararr \end{aligned}$}
			child{node{$-p$}}          
			child{node[PIA]{$\begin{aligned} +\ararr \end{aligned}$}
				child{node{$-q$}}
				child{node{$+r$}}
			};
			\node at (0,0) {$\le$}; 
			
			\node[Ske] at (4,0) {$\begin{aligned} -\wedge \end{aligned}$}
			child {node[Ske] {$\begin{aligned} -\ararr \end{aligned}$}
				child{node[PIA]{$\begin{aligned} +\ararr \end{aligned}$}
					child{node{$-p$}}
					child{node[draw]{$+q$}}
				}
				child{node[Ske]{$\begin{aligned} -\ararr \end{aligned}$}
					child{node[PIA]{$\begin{aligned} +\wbox \end{aligned}$}
						child{node[draw]{$+p$}}
					}
					child{node[draw]{$-r$}}
				}            
			}
			child{node[PIA]{$\begin{aligned} -\wdia \end{aligned}$}
				child{node[draw]{$-r$}}
			}
			;
			\node[rotate = +90] at (3.7, -3.5) {$\underbrace{\hspace{1.3cm}}$};
			\node at (4.1,-3.5) {$\textcolor{blue}{\alpha_p}$};
			\node[rotate = +90] at (5.7, -1.5) {$\underbrace{\hspace{1.3cm}}$};
			\node at (6.1,-1.5) {$\textcolor{red}{\beta_{r1}}$};
			\node[rotate = -90] at (0.5, -2.5) {$\underbrace{\hspace{1.3cm}}$};
			\node at (0.1,-2.5) {$\textcolor{blue}{\alpha_q}$};
			\node[rotate = -90] at (-5, -1) {$\underbrace{\hspace{2.6cm}}$};
			\node at (-5.3,-1) {$\bgamma$};
			
			\node at (5.4,-3) {$\ \}\ \textcolor{red}{\beta_{r2}}$};
			
			%\draw[help lines] (-4,-5) grid (6,6);
			%\node[rotate = +90] at (3.3, -2) {$\underbrace{\hspace{2.6cm}}$};
			%\node at (3.6,-2) {$\beta$};
			\end{tikzpicture}
			
		};
		
	\end{tikzpicture}
	%\caption{Signed generation tree for $\Box (p \rightarrow  q) \to \Box (\Box p\to\Box q)$}
	%\label{fig:fisher-servi}
\end{center}
%In the pictured inequalities above, the squared variable occurrences are the $\varepsilon$-critical ones, the doubly circled nodes are Skeleton, and the single-circle ones are PIA. 		
\end{example}

\subsection{The algorithm SCAN} \label{subsec: SCAN}

The algorithm SCAN \cite{GaOh92c} eliminates second-order quantifiers over predicate variables in formulae of type $\exists P_1,\ldots\exists P_n \varphi$, where $\varphi$ is an arbitrary formula of first-order predicate logic. The algorithm SCAN comprises the following three steps:

\begin{itemize}
    \item The matrix $\varphi$ of input formula $\alpha$ is transformed into skolemized clause form. This means that after the first step the input formula has the form
    \[\exists P_1\ldots\exists P_n \exists f_1\ldots \exists f_m \varphi',\]
    where the $f_i$s are the Skolem functions and $\varphi'$ is a set of clauses.
    \item The resolution and factorization rules involving the predicate variables $P_1,\ldots,P_n$  must be applied until no new clauses can be generated. the resolution and factorization rules are
    \begin{center}
        \begin{tabular}{c c}
\AXC{$P(s_1,\ldots,s_n) \vee C$}
\AXC{$\neg P(t_1,\ldots,t_n) \vee D$}
\BIC{$C \vee D \vee s_1 \neq t_1 \vee \cdots \vee s_n\neq t_n$}
\DP
&
\AXC{$P(s_1,\ldots,s_n) \vee P(t_1,\ldots,t_n) \vee C$}
\UIC{$P(s_1,\ldots,s_n) \vee C \vee s_1 \neq t_1 \cdots \vee s_n\neq t_n$}
\DP
\end{tabular}
    \end{center}
    respectively. We require that no self-resolution is possible (the two clauses involved in a resolution step must be different). As soon as all the possible clauses involving a literal $P$ are generated, delete all the clauses containing $P$. At this step, the algorithm might fail to terminate, in case there are `vicious loops' between some of the clauses (cf.~Example \ref{exe:SCAN_loops}).
    \item Suppose the last step terminates. If there are no clauses left, the input formula is a tautology, while if an empty clause is generated, the input formula is a contradiction. In all the other cases, to obtain a legitimate first-order formula we need to reverse the skolemization. Call the resulting formula SCAN($\alpha$). It is not always possible to unskolemize the output, meaning that in this last step the algorithm could fail to produce a first-order formula equivalent to the input formula.
\end{itemize}

The algorithm is not guaranteed to terminate, but when it does, its output is a first-order formula equivalent to its second-order input formula.

\begin{theorem}\label{th:SCAN_sound}
    If SCAN terminates for a formula $\alpha$ then $\alpha$ is equivalent to SCAN($\alpha$) \cite[Theorem 2.2]{GaOh92c}
\end{theorem}

We conclude this subsection with a useful example of how the algorithm SCAN could be used to compute the first-order correspondent of a modal formula. First Recall the definition of standard translation.

\begin{definition} \label{def:standard_translation}
Given a modal formula $\varphi$ and a variable $x$, the standard translation $\mathrm{ST}_x(\varphi)$ of $\varphi$ is defined as follows.
\[\mathrm{ST}_x(\neg \varphi) = \neg \mathrm{ST}_x(\varphi)\]
\[\mathrm{ST}_x(\varphi \wedge \psi) = \mathrm{ST}_x(\varphi) \wedge \mathrm{ST}_x(\psi) \quad \mathrm{ST}_x(\varphi \vee \psi) = \mathrm{ST}_x(\varphi) \vee \mathrm{ST}_x(\psi)\]
\[\mathrm{ST}_x(\wdia \varphi) = \exists y (xRy \wedge \mathrm{ST}_y(\varphi)) \quad \mathrm{ST}_x(\wbox \varphi) = \forall y (xRy \rightarrow \mathrm{ST}_y(\varphi))\]

\end{definition}
Now we show how SCAN computes the first-order correspondent of $\wdia \wdia P \rightarrow \wdia P$.
\begin{example} \label{exe:SCAN}
    The formula $\wdia \wdia p \rightarrow \wdia p$ is equivalent to $\forall P \forall x \mathrm{ST}_x(\wdia \wdia P \rightarrow \wdia P)$. To use SCAN, we consider the negated $\exists P \neg \forall x \mathrm{ST}_x(\wdia \wdia p \rightarrow \wdia p)$, to be negated back again once SCAN is finished.
    \[\neg \forall x \mathrm{ST}_x(\wdia\wdia p \rightarrow\wdia p) \equiv \exists x\exists y \exists z (xRy \wedge yRz \wedge P(z)) \wedge \forall w (\neg xRw \vee \neg P(w))\]
    The set of clauses corresponding to the expression above is
    \[\{xRy\},\{yRz\},\{P(z)\},\{\neg xRw, \neg P(w)\},\]
    where $x$,$y$, and $z$ are Skolem constants and $w$ is the only universally quantified variable. The only possible resolution step generates the clause $\{\neg xRw, z \neq w\}$, and after the clauses containing $P$ are deleted, we obtain the formula
    \[\exists x\exists y\exists z(xRy \wedge yRz \wedge \forall w(\neg xRw \vee z\neq w)).\]
    Once negated and simplified, the formula above is $\forall x \forall y \forall z (xRy \wedge yRz \rightarrow xRz)$, as expected.
\end{example}

%% file: paper_MASSA/FirstOrderCorrespondentsAsInferenceRules.tex
\section{The algorithm \Algname}
\label{sec:algorithm_MASSA}

In this section, we describe the algorithm MASSA. The steps (i)-(iv) generate the analytic labelled rule $r$ associated with the input formula $\varphi$. Step (v) describes how to read off the geometric implication   from the rule $r$. It is advised to read this section in parallel with Section \ref{sec: examples_MASSA}, where some useful examples are illustrated.

\begin{notation} \label{not:MASSA_definite}
To work properly, the algorithm \Algname must have as input formulas that do not contain $\wedge$ connectives (resp.~$\vee$) with positive polarity (resp.~negative) in its Skeleton. Accordingly, from now on we will assume without loss of generality that any formula given as input to \Algname has been properly preprocessed as explained in Lemma \ref{lemma: reduction to definite}.
\end{notation}

We will now proceed to describe the algorithm in detail.

\begin{enumerate}

\item \textbf{Decomposition of the identity sequent.} \ For any modal formula $\varphi$, consider the identity end-sequent \,$x : \rvarphi \vd x : \bvarphi$\, where the formula in precedent position is coloured red and the formula in succedent position is coloured blue.  
Let $\pi_\varphi$ be a derivation of \,$x : \rvarphi \vd x : \bvarphi$\, obtained by applying the procedure described in the proof of Lemma \ref{lemma: phi implies phi} exhaustively, until there are no logical connectives in every initial sequent. Every new label introduced proceeding bottom-up must be fresh in the entire proof (and not just in every branch).\footnote{The latter requirement guarantees that all the relevant information contained in the end-sequent is maintained (and exploited in rule form) in  $\pi_\varphi$.} At each rule application in $\pi_\varphi$, propagate the colour of the principal formula to the auxiliary formulas. To make the derivation visually less cluttered, we apply the tonicity rules in their noninvertible form (cf.~Section \ref{sec:LabelledRelationalSequentCalculi}).

\item \textbf{Atomic cuts + PIA parts.} \ Consider the leaves of $\pi_\varphi$ and perform all possible cuts on atomic red-coloured formulas $x: \rp$ occurring in $\pi_\varphi$. These cuts generate new axioms of the form\, $\Gamma, y= z, y : \bp \vd z : \bp, \Delta$\, in which the new relational atom $y = z$ appears in  the conclusion of each cut with cut formulas $y : \rp$ and $z : \rp$.  If a proposition variable $x: \rp$ occurs only positively or only negatively in $\varphi$,  then cut either with an atomic initial rule of the form $ x: \abot \vd x: \rp$ or with $x: \rp \vd x: \top$.  Collect all the conclusions of these cut-applications, and use them as leaves in a (cut-free) forward-chaining proof-search with goal $\vd x : \bvarphi$. Stop the proof search only when you reach a sequent where all its formula are some of the maximal PIA subformulas\footnote{Notice that he maximal PIA subformulas here coincide with those subformulas that can be constructed using only tonicity rules.} of $\varphi$, and possibly formulas of the type $y : \bot$ or $y : \top$. Collect all the attempts $\pi^{\,i}_\varphi$ generated in this way.

\item \textbf{Skeleton part.} \ Perform a backward-chaining proof search on $\vd x : \bvarphi$ in which we decompose all and only the Skeleton connectives of $\varphi$.\footnote{Here we are compositionally destroying all the Skeleton connectives namely $\wdia$ and $\aand$ if occurring in precedent position, and $\wbox, \aor$ and $\ararr$ if occurring in succedent position. Decomposing the Skeleton of $\varphi$ is equivalent to apply only translation rules proceeding bottom-up.}
Notice that when the input formula is definite (cf.~Notation \ref{not:MASSA_definite}), we only need unary rules to decompose the Skeleton of $\varphi$.

\item \textbf{Skeleton-PIA merging.} A \emph{merging point} is a tuple of sequents $(S_1, ..., S_n, S)$, of which the $S_i$ are the endsequents of all the proof-trees  $\pi^{\, i}_\varphi$ generated in item (ii) and $S$ is the uppermost sequent of   the proof-section generated in item (iii). We call the $S_i$ the {\em premises} and $S$ the {\em conclusion} of the merging point. If $(S_1, ..., S_n, S)$ is a merging point, then it is an application of the rule $r$ in output, which provides the missing step in the derivation of $\vd x : \bvarphi$.

Let $R_i$ and $R$ be the relational parts of $S_i$ and $S$ respectively. The  rule $r$ associated with the merging point is:
\begin{center}
\AXC{$R_1, \Gamma_1 \vd \Delta_1$}
\AXC{...}
\AXC{$R_n, \Gamma_n \vd \Delta_n$}
\LL{$r $}
\TIC{$R, \Gamma_1, \ldots ,\Gamma_n \vd \Delta_1, \ldots, \Delta_n$}
\DP
\end{center}

\item \textbf{Reading off the geometric axiom from the rule.} 
Let $F_i$ be defined as the conjunction of the relational atoms in $R_i$ in case  in $S_i$ there are no occurrences of $y : \bot$ in precedent position or of $y : \top$ in succedent position (an empty conjunction will be regarded as $\top$). Otherwise, let $F_i$ be $\bot$. If in $S$ there are formulas $y : \bot$ (resp. $y : \top$) in precedent (resp.~succedent) position, then the required geometric formula is $\top$. Otherwise, the  geometric axiom which we can read off from the rule $r$ is:
\begin{center}
$\forall \ol{x} [\bigwedge R(\ol{x}) \ararr \bigvee_i \ol{\exists y_i}F_i(\ol{x}, \ol{y_i}))].$
\end{center}

\end{enumerate}

%\begin{rem}
%$Notice that sequent \,$x : \rvarphi \vd x : \bvarphi$\, considered at step (i) is an identity, so it is always derivable in the basic calculus by means of logical (invertible) rules. 
%, plus display postulates if we consider the fully residuated language (and we use display calculi proofs).
%Moreover, we observe that the (red/blue) \emph{atomic flows} \cite{xxx} --namely, the history-graphs of atomic propositional occurrences-- are \emph{trees} if we use additive rules for the propositional connectives and multiplicative rules for the modalities, and they are \emph{chains} if we use only multiplicative rules. \red{Maybe this remark should be deleted.}
%\end{rem}

 Steps (i) and (ii) can be intuitively justified as follows. Whenever  $\varphi$ is a theorem of $K$, the calculus G3K derives $\vd x : \varphi$ without any additional rule. Otherwise, we need to identify some assumptions $\Gamma$ which allow us to derive $\Gamma \vd x : \varphi$.  Clearly, the \emph{minimal} set of assumptions $\Gamma$ under which  $\varphi$ is derivable  is $\Gamma = \{x : \rvarphi\}$. Then, at step (i), we equivalently transform the additional assumption $x : \rvarphi$ into pure relational information and also information stored in the atomic propositions of the form $x:\rp$. The cuts performed in step (ii) extract additional pure relational information from these atomic propositions. 

\section{Examples} 
\label{sec: examples_MASSA}

In the present section, we illustrate  the algorithm \Algname by running it on some definite analytic inductive  formulas. % of Example \ref{examples}.
Let us start with $\wdia \wbox A \rightarrow \wbox \wdia A$, the Church-Rosser or directedness axiom (cf.~Example \ref{examples} (iii)). %and show how to generate the equivalent rule Dir (as well as its associated geometric axiom which is also its first-order correspondent).

\textbf{Step (i).} We build the proof $\pi_\varphi$:

\begin{center}
{\footnotesize
\AXC{$ \ $}
\RL{Id$_{t:A}$}
\UIC{$(7.1) \ xRy, yRt, t : \bAu \vd t : \rAt$}
\LL{$\wbox_L$}
\UIC{$(6.1) \ xRy, yRt, y : \wbox \bAu \vd t : \rAt$}
\RL{$\wbox_R$}
\UIC{$(5.1) \ xRy, y : \wbox \bAu \vd y : \wbox \rAt$}
\RL{$\wdia_R$}
\UIC{$(4.1) \ xRy, y : \wbox \bAu \vd x : \wdia \wbox \rAt$}
\LL{$\wdia_L$}
\UIC{$(3.1) \ x : \wdia \wbox \bAu \vd x : \wdia \wbox \rAt$}
\AXC{$ \ $}
\RL{Id$_{w:A}$}
\UIC{$(7.2) \ xRz, zRw, w : \rAq \vd w : \bAd$}
\RL{$\wdia_R$}
\UIC{$(6.2) \ xRz, zRw, w : \rAq \vd z : \wdia \bAd$}
\LL{$\wdia_L$}
\UIC{$(5.2) \ xRz, z : \wdia \rAq \vd z : \wdia \bAd$}
\LL{$\wbox_L$}
\UIC{$(4.2) \ xRz, x : \wbox \wdia \rAq \vd z : \wdia \bAd$}
\RL{$\wbox_R$}
\UIC{$(3.2) \ x : \wbox \wdia \rAq \vd x : \wbox \wdia \bAd$}
\LL{$\ararr_L$}
\BIC{$(2) \ x : \wdia \wbox \bAu, x : \wdia \wbox \rAt \ararr \wbox \wdia \rAq \vd x : \wbox \wdia \bAd$}
\RL{$\ararr_R$}
\UIC{$(1) \ x : \wdia \wbox \rAt \ararr \wbox \wdia \rAq \vd x : \wdia \wbox \bAu \ararr \wbox \wdia \bAd$}
\DP
}
\end{center}

\textbf{Step (ii).} %We want to get rid of the assumption $x : \rvarphi$. 
We consider the leaves $(7.1)$ and $(7.2)$ and  perform all the atomic cuts on red coloured formulas. 

\begin{center}
{\footnotesize
\AXC{ \ }
\RL{Id$_{t:A}$}
\UIC{$(7.1) \ xRy, yRt, t : \bAu \vd t : \rAt$}
\AXC{ \ }
\RL{Id$_{w:A}$}
\UIC{$(7.2) \ xRz, zRw, w : \rAq \vd w : \bAd$}
\RL{Cut($\rAt, \rAq$)}
\BIC{$xRy, yRt, xRz, zRw, t = w; t : \bAu \vd w : \bAd$}
\DP
}
\end{center}

We now construct the upper portion of the proof $\pi^1_\varphi$.\footnote{Notice that we could also construct a proof with a different order of rule applications (e.g.~in this case, proceeding top down, first we apply $\wbox_L$ and then $\wdia_R$). Such trivial permutations of rules generate, strictly speaking, different syntactic proofs but do not change the merging point. So, it is enough to pick one of those proofs.} In this step, we  build up the PIA sub-formulas of $\varphi$.

\begin{center}
{\small
\AXC{$\pi_\varphi^1$}
\UIC{$xRy, yRt, xRz, zRw, t = w, t : \bA \vd w : \bA$}
\RL{$\wdia_R$}
\UIC{$xRy, yRt, xRz, zRw, t = w, t : \bA \vd z : \wdia \bA$}
\LL{$\wbox_L$}
\UIC{$xRy, yRt, xRz, zRw, t = w, y : \wbox \bA \vd z : \wdia \bA$}
\dashedLine
\UIC{$ \ $}
\DP
}
\end{center}

\textbf{Step (iii).} In this step, we  work on the Skeleton  of $\varphi$.

\begin{center}
\AXC{}
\dashedLine
\UIC{$xRz, xRy \,, y : \wbox \bA \vd z : \wdia \bA$}
\LL{$\wdia_L$}
\UIC{$xRz \,, x : \wdia \wbox \bA \vd z : \wdia \bA$}
\RL{$\wbox_R$}
\UIC{$x : \wdia \wbox \bA \vd x : \wbox \wdia \bA$}
\RL{$\ararr_R$}
\UIC{$\vd x : \wdia \wbox \bA \ararr \wbox \wdia \bA$}
\DP
\end{center}

\vspace{0.2cm}

\textbf{Step (iv).} We now reach a merging point, and hence generate the  rule Dir:

\begin{center}
\begin{tabular}{c}
\AXC{}
\UIC{$xRy, yRt, xRz, zRw, t = w, t : \bA \vd w : \bA$}
\RL{$\wdia_R$}
\UIC{$xRy, yRt, xRz, zRw, t = w, t : \bA \vd z : \wdia \bA$}
\LL{$\wbox_L$}
\UIC{$xRy, yRt, xRz, zRw, t = w, y : \wbox \bA \vd z : \wdia \bA$}
\dashedLine
\LL{Dir}
\UIC{$xRy, xRz \,, y : \wbox \bA \vd z : \wdia \bA$}
\LL{$\wdia_L$}
\UIC{$xRz \,, x : \wdia \wbox \bA \vd z : \wdia \bA$}
\RL{$\wbox_R$}
\UIC{$x : \wdia \wbox \bA \vd x : \wbox \wdia \bA$}
\RL{$\ararr_R$}
\UIC{$\vd x : \wdia \wbox \bA \ararr \wbox \wdia \bA$}
\DP
\end{tabular}
\end{center}

\textbf{Step (v).} Finally, the FO-correspondent reads 
\begin{center}
    $\forall x \forall y \forall z [xRy \wedge xRz \ararr \exists t \exists w (yRt \wedge zRw \wedge t = w)],$
\end{center}
which is equivalent to directedness. 

%Let us proceed with more examples.
For the next example, let us execute MASSA on the `functionality' axiom (cf.~Example \ref{examples} (i)). The pruned proof-tree generated in  the first step is the following:

\begin{center}
\AXC{}
\LL{Id$_{y:A}$}
\UIC{$xRy, y: \bA \vd y: \rA$}
\LL{$\wdia_R$}
\UIC{$xRy, y: \bA \vd x: \wdia \rA$}
\LL{$\wdia_L$}
\UIC{$x: \wdia \bA \vd x: \wdia \rA$}

\AXC{}
\RL{Id$_{z:A}$}
\UIC{$xRz, z: \rA \vd z: \bA$}
\RL{$\wbox_L$}
\UIC{$xRz, x: \wbox \rA \vd z: \bA$}
\RL{$\wbox_R$}
\UIC{$x: \wbox \rA \vd x: \wbox \bA$}

\RL{$\ararr_L$}
\BIC{$ x : \wdia \rA \ararr \wbox \rA, x: \wdia \bA \vd x: \wbox \bA$}
\RL{$\ararr_R$}
\UIC{$ x : \wdia \rA \ararr \wbox \rA \vd x: \wdia \bA \ararr \wbox \bA$}
\DP
\end{center}

The leaves on which we perform the only possible cut are written below: 
\begin{center}
$xRy, y: \bA \vd y: \rA \quad xRz, z: \rA \vd z: \bA.$
\end{center}
After performing step (ii) and (iii), the merging point is reached, which generates the following derivation and rule (step (iv)):

%$$y = z, xRy, xRz, y: \bA \vd z: \bA$$

\begin{center}
\AXC{$ \ $}
\UIC{$y = z, xRy, xRz, y: \bA \vd z: \bA$}
\dashedLine
\LL{Fun}
\UIC{$xRy, xRz, y: \bA \vd z: \bA$}
\LL{$\wdia_L$}
\UIC{$xRz, x: \wdia \bA \vd z: \bA$}
\RL{$\wbox_R$}
\UIC{$x: \wdia \bA \vd x: \wbox \bA$}
\RL{$\ararr_R$}
\UIC{$\vd x: \wdia \bA \ararr \wbox \bA$}
\DP
\end{center}

from which the  first-order correspondent (step (v)) below can be read off:

\begin{center}
$\forall x \forall y \forall z (xRy \land xRz \ararr y = z)$.
\end{center}

Merging points do not need to be unary. To see this, let us consider the formula $\wbox (\wbox A \ararr B) \lor \wbox (\wbox B \ararr A)$ (cf.~Example \ref{examples} (v)). After performing step (i), the leaves of $\pi$ are as follows:
\begin{center}
$xRy, yRz, z: \bA \vd z: \rA \quad\quad xRt, t: \rA \vd t : \bA$ $xRt, tRw, w: \bB \vd w : \rB \quad\quad xRy, y: \rB \vd y: \bB$
\end{center}

After performing steps (ii) and (iii), we generate a binary merging point and we provide the following derivation (step (iv)), obtaining the rule $3$: 

\begin{center}
{\small
\AXC{$ \ $}
\UIC{$xRy, yRz, xRt, z = t, z: A \vd t: A$}
\UIC{$xRy, yRz, xRt, z = t, y: \wbox A \vd t: A$}

\AXC{$ \ $}
\UIC{$xRt, tRw, xRy, y = w, w: B \vd y: B$}
\UIC{$xRt, tRw, xRy, y = w, t: \wbox B \vd y: B$}

\dashedLine
\LL{$3$}
\BIC{$xRy, xRt, y: \wbox A, t: \wbox B \vd y: B, t: A$}
\UIC{$xRy, xRt, y: \wbox A \vd y: B, t: \wbox B \ararr A$}
\UIC{$xRy, xRt \vd y: \wbox A \ararr B, t: \wbox B \ararr A$}
\UIC{$xRy \vd y: \wbox A \ararr B, x: \wbox (\wbox B \ararr A)$}
\UIC{$\vd x: \wbox (\wbox A \ararr B), x: \wbox (\wbox B \ararr A)$}
\UIC{$\vd x: \wbox (\wbox A \ararr B) \lor \wbox (\wbox B \ararr A)$}
\DP
}
\end{center}

The first order correspondent (step (v)) reads
\begin{center}
$\forall x \forall y \forall t (xRy \land xRt \ararr \exists z (yRz \land z = t) \lor \exists w (tRw \land y = w))$
\end{center}
which is equivalent to
\begin{center}
$\forall x \forall y \forall t (xRy \land xRt \ararr yRt \lor tRy).$
\end{center}

The examples discussed so far are all Sahlqvist. However, MASSA is successful on  (definite analytic) formulas which are {\em properly inductive}, such as the axiom $K := \wbox (A \ararr B) \ararr (\wbox A \ararr \wbox B)$. After performing step (i), the leaves of $\pi_K$ are as follows:
\begin{center}
$xRz, z: \bA \vd z: \rA \quad\quad  xRy, y: \rA \vd y: \bA$ $xRy, y: \bB \vd y: \rB \quad\quad  xRt, t: \rB \vd t: \bB$
\end{center}

After performing steps (ii) and (iii), we reach a merging point and hence the rule deriving $K$ as follows (step (iv)):

\begin{center}
{\small
\AXC{$ \ $}
\UIC{$xRz, zRy, y = z, z: A \vd y: A$}
\UIC{$xRz, zRy, y = z, x: \wbox A \vd y: A$}

\AXC{$ \ $}
\UIC{$xRy, xRt, t = y, y: B \vd t: B$}

\BIC{$xRt, xRy, xRz, t = y, t = z, y: A \ararr B, x: \wbox A \vd t: B$}
\UIC{$xRt, xRy, xRz, t = y, t = z, x: \wbox (A \ararr B), x: \wbox A \vd t: B$}
\dashedLine
\LL{$K$}
\UIC{$xRt, x: \wbox (A \ararr B), x: \wbox A \vd t: B$}
\UIC{$x: \wbox (A \ararr B), x: \wbox A \vd x: \wbox B$}
\UIC{$x: \wbox (A \ararr B) \vd x: \wbox A \ararr \wbox B$}
\UIC{$\vd x: \wbox (A \ararr B) \ararr (\wbox A \ararr \wbox B)$}
\DP
}
\end{center}

The first-order correspondent (step (v)) reads
\begin{center}
    $\forall x \forall t (xRt \ararr \exists y \exists z (xRy \wedge xRz \wedge t = y \wedge t = z))$
\end{center}
which is equivalent to $\aatop$ as expected, since the input formula $K$ is derivable in G3K, i.e.~is valid in every Kripke frame.

The next example investigates the case where there is a uniform propositional atom. Consider the formula $\wbox (A \vee B)\rightarrow \wdia A$. After the first phase, the mismatched leaf $xRy, y: \bB \vd y: \rB$ becomes $xRy, y: B \vd y: \top$ at the beginning of the second phase.

\begin{center}
{\small
\AXC{$ \ $}
\UIC{$xRy, y: B \vd y: \top$}

\AXC{$ \ $}
\UIC{$xRy, xRz, y = z, y: A \vd z: A$}
\UIC{$xRy, xRz, y = z, y: A \vd x: \wdia A$}

\BIC{$xRy, xRz,y = z, y: (A \vee B) \vd x: \wdia A, y: \top$}
\UIC{$xRy, xRz,y = z, x: \wbox (A \vee B) \vd x: \wdia A, y: \top$}
\dashedLine
\UIC{$ x: \wbox (A \vee B) \vd x: \wdia A$}
\UIC{$\vd x: \wbox (A \vee B)\rightarrow \wdia A$}
\DP
}
\end{center}

Since $y:\top$ appears in the succedent of the sequent over the merging point, the first-order correspondent of the formula is $\bot$.

As a last example, consider now the formula $ \wdia A \wedge A \rightarrow \wbox A \vee \wbox \wdia A$. At the end of the first step, we obtain the four leaves

\begin{center}
\begin{tabular}{rcl}
$xRy, y: \bA \vd y: \rA$ & & $xRw, w: \rA \vd w: \bA$ \\ \\
$x: \bA \vd x: \rA$ & & $xRz, zRt, t: \rA \vd t: \bA$
\end{tabular}
\end{center}

We need to perform all the possible cuts between the four leaves above for a total of four cuts. In the end we obtain the equivalent rule $R$

\begin{center}
{\small
\AXC{$\pi_1$}
\AXC{$\pi_2$}
\AXC{$\pi_3$}
\AXC{$\pi_4$}
\dashedLine
\LL{$R$}
\QIC{$xRy, xRw, xRz, y: A, x: A \vd w: A, z:  \wdia A$}
\UIC{$xRy, xRw, y: A, x: A \vd w: A, x:\wbox \wdia A$}
\UIC{$xRy, y: A, x: A \vd x: \wbox A, x:\wbox \wdia A$}
\UIC{$ x: \wdia A, x: A \vd x: \wbox A, x:\wbox \wdia A$}
\UIC{$ x: \wdia A, x: A \vd x: \wbox A \vee \wbox \wdia A$}
\UIC{$ x: \wdia A \wedge A \vd x: \wbox A \vee \wbox \wdia A$}
\UIC{$ \vd x: \wdia A \wedge A \rightarrow \wbox A \vee \wbox \wdia A$}
\DP
}
\end{center}

where $\pi_1$, $\pi_2$, $\pi_3$, and $\pi_4$ are the leaves

\[xRy, xRw, y=w, y:A\vd w:A\]
\[xRy, xRz, zRt, y=t, y:A\vd t:A\]
\[xRw, x=w, x:A\vd w:A\]
\[xRz, zRt, x=t, x:A\vd t:A\]
respectively. The first-order correspondent is

\[\forall x \forall y \forall w \forall z [xRy \wedge xRw \wedge xRz \rightarrow y=w \vee \exists t(t = y \wedge zRt) \vee x=w \vee \exists t (t=x \wedge zRt)]\]

%% file: paper_MASSA/NewDevelopments.tex
\section{SCAN is successful on analytic inductive axioms} \label{sec:SCAN_successful}

In this section, we prove that the algorithm SCAN (cf.~Subsection \ref{subsec: SCAN}) terminates when its input is a definite analytic inductive axiom. The proof differs substantially from the one given in \cite{Goranko:SCAN:Complete}, since we need to deal with a different `source' of non-termination (cf.~Remark \ref{rem:different_from_Goranko} below). 

\begin{notation}
From now on, we will sometimes write the inequality $\varphi \leq \psi$ to denote the formula $\varphi \rightarrow \psi$.
\end{notation}

We begin with a useful lemma that relates the maximal PIA subformulas of the definite analytic inductive axiom in input with some disjunctive clauses of the standard translation of its negation in conjunctive normal form.

\begin{lemma} \label{lem:clausification}
    Consider a definite analytic inductive inequality $\varphi \leq \psi$. After translating $\mathrm{ST}_x(\varphi) \ \&\ \neg\mathrm{ST}_x(\psi)$ in conjunctive normal form, there will exist a bijection between its clauses containing at least one second-order predicate symbol and the maximal PIA subformulas of $\varphi \leq \psi$. The bijection is such that all the instances of a propositional atom match the instances of the corresponding second-order predicate symbols, but with reversed polarity. 
\end{lemma}
\begin{proof}
Proceeding by induction on the structural complexity of $\varphi$ and $\psi$ one proves that, once Skolemized, $\mathrm{ST}_x(\varphi)\ \&\ \neg\mathrm{ST}_x(\psi)$ is equivalent to
\[\bigwith_i \mathrm{ST}_{x_i}(\alpha_i)\ \&\ \bigwith_j\neg\mathrm{ST}_{y_j}(\beta_j) \ \&\ \mathcal{R},\]
where the $\alpha_i$s (resp.~$\beta_j$s) are the maximal positive (resp.~negative) PIA subformulas of $\varphi \leq \psi$ and $\mathcal{R}$ is a conjunction of relational atoms. Notice that during this first phase, no universal quantifier appears, meaning that all the symbols introduced via Skolemization are constant symbols (cf.~Example \ref{exe:SCAN_translation_PIA} below).

Again via structural induction, it is shown that every $\mathrm{ST}_{x_i}(\alpha_i)$ (resp.~$\neg\mathrm{ST}_{y_j}(\beta_j)$) is equivalent to a universally quantified disjunction containing exactly the same atomic propositions of $\alpha_i$ (resp.~$\beta_j$) translated as second-order predicate symbols (possibly together with first-order literals), meaning that it gets translated to a clause, as we wanted to prove. Notice that during this last phase, no existential quantifier is introduced. The last claim follows from the observation that $\mathrm{ST}_x(\varphi) \ \&\ \neg\mathrm{ST}_x(\psi)$ is just the translation of $\varphi \nleq \psi$, hence the polarity of all the atomic proposition instances are reversed.
\end{proof}

We give a concrete example to better visualize the proof of the preceding lemma.

\begin{example} \label{exe:SCAN_translation_PIA}
Consider the definite analytic inductive (but not sahlqvist) formula 
\[\wdia(B \wedge \wbox A) \rightarrow \wdia A \vee \wbox \wdia(\wdia B \wedge A)\]
whose maximal PIA subformulas are $B$ (the leftmost instance), $\wbox A$, $\wdia A$, and $\wdia (\wdia B \wedge A)$. In this example we do not distinguish between lattice connectives and meta-conjunctions/disjunctions. As described in the proof above, we start computing the first-order translation $\mathrm{ST}_x(\wdia (B \wedge \wbox A)) \wedge \neg\mathrm{ST}_x(\wdia A \vee \wbox\wdia(\wdia B \vee A))$:
\[\mathrm{ST}_x(\wdia (B \wedge \wbox A)) = \exists y (xRy \wedge \mathrm{ST}_y(B)\wedge \mathrm{ST}_y(\wbox A)) \]
\[\neg\mathrm{ST}_x(\wdia A \vee \wbox\wdia(\wdia B \vee A)) = \neg\mathrm{ST}_x(\wdia A) \wedge \exists x (xRz \wedge \neg\mathrm{ST}_z(\wdia(\wdia B \wedge A)))\]
After Skolemization, we get
\[\mathrm{ST}_y(B) \wedge \mathrm{ST}_y(\wbox A) \wedge \neg \mathrm{ST}_x(\wdia A) \wedge \neg\mathrm{ST}_z(\wdia(\wdia B \wedge A)) \wedge xRy\wedge xRz\]
as expected. The intuition behind the result is that when an analytic inductive inequality is definite, its Skeleton contains only conjunctions (resp.~disjunctions) and diamond-like (resp.~box-like) connectives in the antecedent (resp.~succedent). We now translate and clausify every maximal PIA subformula.
\[\mathrm{ST}_y(B) = B(y) \quad \mathrm{ST}_y(\wbox A) = \forall w(\neg yRw \vee A(w))\]
\[\neg \mathrm{ST}_x(\wdia A) = \forall w'(\neg xRw' \vee \neg A(w'))\]
\[\neg\mathrm{ST}_z(\wdia(\wdia B \wedge A))= \forall t_1 \forall t_2 (\neg zRt_1 \vee \neg t_1Rt_2 \vee \neg A(t_1) \vee \neg B(t_2))\]
The intuition behind this last step is that when an analytic inductive inequality is definite, its PIA contains only conjunctions (resp.~disjunctions) and diamond-like (resp.~box-like) connectives when they have positive (resp.~negative) polarity.
\end{example}

After the previous example and a further comparison with Example \ref{exe:SCAN}, we elaborate on a couple of observations that will be useful later on, when proving the soundness of MASSA (cf.~Proposition \ref{prop: MASSA_soundness_via_SCAN}).

\begin{remark} \label{rem:SCAN_MASSA_comparison}
    Once translated in clause normal form, every negated definite analytic inductive formula  $\mathrm{ST}_x(\varphi)\ \&\ \neg\mathrm{ST}_x(\psi)$ containing $n$ Skeleton connectives and $m$ maximal PIAs will be of the form
    \[\{x_1Ry_1\},\ldots,\{x_nRy_n\},\]
    \[\{\neg z_{1,1}Rw_{1,1},\ldots,\neg z_{1,k_1}Rw_{1,k_1},*P_{1,1}(t_{1,1}),\ldots,*P_{1,k'_1}(t_{1,k'_1})\},\]
    \[\vdots\]
    \[\{\neg z_{m,1}Rw_{m,1},\ldots,\neg z_{m,k_m}Rw_{m,k_m},*P_{m,1}(t_{m,1}),\ldots,*P_{m,k'_m}(t_{m,k'_m})\}.\]
    Where the first $n$ clauses come from the translation of the Skeleton portion of the input formula and all the variables $x_i$s and $y_i$s are existentially quantified. The last $m$ clauses are the translation of the maximal PIAs, where $*P$ could be either $P$ or $\neg P$, and all their variables that are not already in $x_1,y_1,\ldots,x_n,y_n$ are universally quantified.
\end{remark}

With the next corollary, we prove that the success of SCAN on analytic inductive inequalities reduces to checking the absence of infinite loops.

\begin{corollary} \label{cor:reverse_skolem}
    When computing the conjunctive normal form of $\mathrm{ST}_x(\varphi) \ \&\ \neg\mathrm{ST}_x(\psi)$ corresponding to the negation of the definite analytic inductive inequality $\varphi \leq \psi$, the Skolemization adds only constant symbols, therefore, to prove that SCAN is successful on analytic inductive axioms it is sufficient to prove that it does not loop, the un-Skolemization step being always successful.
\end{corollary}
\begin{proof}
    From the proof of Lemma \ref{lem:clausification} it is clear that, during the translation, no existential quantifier is introduced under the scope of a universal quantifier. This property follows from the fact that in an analytic inductive inequalities all branches are good.
\end{proof}

Before stating our main result, we show an example where SCAN's resolution phase loops forever.

\begin{example} \label{exe:SCAN_loops}
    Consider the clauses $\mathcal{C}_1 \coloneqq \{A(x), B(y)\}$ and $\mathcal{C}_2\coloneqq\{\neg A(z), \neg B(w)\}$, we show three consecutive resolution steps:

\begin{center}
\begin{tikzpicture}
  \begin{scope}[rotate=180]
    \node {$\{z\neq x,z\neq w,B(y),\neg B(w)\}$}
      [level distance=2.5cm, sibling distance=3.5cm,
       every node/.style={ align=center}]
      child {node {$\{z\neq x,y\neq w,B(y), A(x)\}$}
        child {node {$\{z\neq x,B(y),\neg B(w)\}$}
          child {node {$\mathcal{C}_1$}}
          child {node {$\mathcal{C}_2$}}
        }
        child {node {$\mathcal{C}_1$}}
      }
      child {node {$\mathcal{C}_2$}};
  \end{scope}
\end{tikzpicture}
\end{center}

It is clear that in this example resolution does not terminate.
    
\end{example}

We now have everything we need to prove the main proposition of this section.

\begin{proposition} \label{prop: SCAN_terminates}
The algorithm SCAN is always successful on definite analytic inductive axioms.
\end{proposition}
\begin{proof}
Consider an arbitrary definite analytic inductive inequality $\varphi \leq \psi$ with propositional atoms $p_1, \ldots, p_n$. To make the proof easier to read, assume without loss of generality that $\Omega$ is $p_1 > p_2 >\ldots> p_n$ and we further assume that $\varepsilon(p_i)=1$ for all $i \in\{1,\ldots,n\}$. The proof can be easily adapted to different order types switching some of the negations in the remainder of the proof below.

By virtue of Corollary \ref{cor:reverse_skolem}, to prove the thesis is sufficient to argue that SCAN does not loop forever during its resolution phase, the unskolemization being unproblematic. Since every clause of the clausification of $\mathrm{ST}_x(\varphi)\ \&\ \neg\mathrm{ST}_x(\psi)$ comes from a maximal PIA subformula of $\varphi \leq \psi$ (cf.~Lemma \ref{lem:clausification}), every clause will be exactly of one of the following types:
\begin{itemize}
    \item $\{p_i(x)\}\cup \bigcup^{n}_{j=i+1}\{\neg p_{j}(x_{j,1}),\ldots,\neg p_j(x_{j,m_j})\}\cup\mathcal{R}$, with $i$ in $\{1,\ldots,n\}$, exactly one instance of $p_i$, zero or more instances of $\neg p_j$ with $i<j\leq n$, and $\mathcal{R}$ a set with only relational literals,
    \item $\bigcup^{n}_{j=1}\{\neg p_j(y_{j,1}),\ldots,\neg p_j(y_{j,m_j})\} \cup\mathcal{R}$, with each $\neg p_j$ of arbitrary multiplicity $m_j$ and $\mathcal{R}$ a set with only relational literals.
\end{itemize}
Clauses from the first type come from maximal PIA subformulas containing a critical occurrence of a propositional variable, while maximal PIA subformulas with no critical occurrences correspond to the second type of clauses. To prove the termination of the resolution step, we will preliminarily assign to every clause an element in the well-ordered set $2^n\times\omega^n$, where the product order is taken lexicographically. The assignment is defined as follows\footnote{Using a slight abuse of notation, we write $k^m$ to denote the $m$-uple $(k,\ldots,k)$, with $k$ we denote the single-element tuple $(k)$, and finally with $(a_1,\ldots,a_n) \times (b_1,\ldots,b_m)$ we denote the tuple $(a_1,\ldots,a_n,b_1,\ldots,b_m)$.}:
\begin{itemize}
    \item the clause $\{p_i(x)\}\cup \bigcup^{n}_{j=i+1}\{\neg p_{j}(x_{j,1}),\ldots,\neg p_j(x_{j,m_j})\}\cup\mathcal{R}$ gets assigned to the element $0^{i-1}\times1\times0^{n-i}\times(0,\ldots,0,m_{i+1},\ldots,m_n)$, where $m_j$ is the multiplicity of $\neg p_j$;
    \item the clause $\bigcup^{n}_{j=1}\{\neg p_j(y_{j,1}),\ldots,\neg p_j(y_{j,m_j})\} \cup\mathcal{R}$ gets assigned to $0^n\times(m_1,\ldots,m_n)$, where $m_j$ is the multiplicity of $\neg p_j$.
\end{itemize}
From now on, the image of a clause under the assignment described above is called its \emph{size}. Our strategy will be to show that only a finite number of new clauses can be created using resolution steps. From the definition of the assignment, we deduce that the size of the output of a resolution step is strictly smaller than the size of at least one of its inputs. To prove our thesis, it is sufficient to argue that the $(n+i)$th coordinate of the size of every new clause has an upper bound, denoted as $\max p_i$. Define $A_{i,j}$ (with $i$ in $\{1,\ldots,n\}$ and $j>i$) as the highest multiplicity of $\neg p_j$ among all clauses of the form $\{p_i(x)\}\cup \bigcup^{n}_{j=i+1}\{\neg p_{j}(x_{j,1}),\ldots,\neg p_j(x_{j,m_
j})\}\cup\mathcal{R}$, and define $B_j$ (with $j$ in $\{1,\ldots,n\}$) as the highest multiplicity of $\neg p_j$ among all clauses of the form $\bigcup^{n}_{j=1}\{\neg p_j(y_{j,1}),\ldots,\neg p_j(y_{j,m_j})\} \cup\mathcal{R}$.
We claim $\max p_1 = B_1$ and $\max p_i \leq \sum^{i-1}_{j=1}(A_{j,i}\max p_j) + B_i$ for $i>1$. To see why $\max p_1 = B_1$, notice that if $\mathcal{C}_1$ and $\mathcal{C}_2$ are the two premise clauses in a resolution step, then one of them must have zero instances of $\neg p_1$. As for the $i>1$ case, suppose we have a correct estimate of $\max p_j$ for $j < i$. To compute an upper bound for $\max p_i$, assume that after a number of resolution steps we obtain a clause $\mathcal{C}$ of the form $\bigcup^{n}_{k=1}\{\neg p_k(y_{k,1}),\ldots,\neg p_k(y_{k,m_k})\} \cup\mathcal{R}$ with exactly $\max p_j$ instances of $\neg p_j$ for $j<i$ and exactly $B_i$ instances of $\neg p_i$ (see Example \ref{exe:SCAN_termination} for a concrete case where our estimated upper bound coincides with the real value of $\max p_i$). For every $j<i$, we resolve $\mathcal{C}$ with a clause $\mathcal{C}_j$ of the type $\{p_j(x)\}\cup \bigcup^{n}_{k=j+1}\{\neg p_{k}(x_{k,1}),\ldots,\neg p_k(x_{k,m_
k})\}\cup\mathcal{R}$ containing $A_{j,i}$ instances of $\neg p_i$, meaning that in every resolution step we add $A_{j,i}$ instances of $\neg p_i$ to $B_i$. We can repeat the process $\max p_j$ times for each $j$, leading to a total of $\sum^{i-1}_{j=1}(A_{j,i}\max p_j) + B_i$ instances of $\neg p_i$ in the `worst' (in the sense of `with more instances of $\neg p_i$') scenario.
\end{proof}

\begin{example} \label{exe:SCAN_termination}
Consider the following set $\{\mathcal{C}_1,\mathcal{C}_2,\mathcal{C}_3,\mathcal{C}_4\}$ of clauses, with $A_{1,2} = A_{1,3} = 2$, $A_{2,4} = 1$, $B_1 = B_2 = 2$, and $B_3 = 1$:
    \begin{itemize}
        \item $\mathcal{C}_1 \coloneqq \{p_1(x_1), \neg p_2(x_2), \neg p_2(x_3), \neg p_3(x_4), \neg p_3(x_5)\} \cup \mathcal{R}_1$;
        \item $\mathcal{C}_2 \coloneqq \{p_2(x_6), \neg p_3(x_7)\} \cup \mathcal{R}_2$;
        \item $\mathcal{C}_3 \coloneqq \{p_3(x_8)\} \cup \mathcal{R}_3$;
        \item $\mathcal{C}_4 \coloneqq \{\neg p_1(x_9), \neg p_1(x_{10}), \neg p_2(x_{11}), \neg p_2(x_{12}), \neg p_3(x_{13})\} \cup \mathcal{R}_4$.
    \end{itemize}
    It is clear that $\max p_1 = 2$ and according to Proposition \ref{prop: SCAN_terminates}, we expect $\max p_2 \leq 6$ and $\max p_3 \leq 11$. To obtain a clause with the maximum number of instances of $\neg p_2$, resolve $\mathcal{C}_1$ against $\mathcal{C}_4$ twice to obtain a clause of the type
    \[\mathcal{C}' = \{\neg p_2,\neg p_2,\neg p_2,\neg p_2,\neg p_2,\neg p_2,\neg p_3,\neg p_3,\neg p_3,\neg p_3,\neg p_3\} \cup \mathcal{R}',\]
    where we omitted the details concerning first order variables to not get distracted with unnecessary details. Finally, to obtain clause with the maximum number of instances of $\neg p_3$, we resolve $\mathcal{C}'$ against $\mathcal{C}_2$ six times to obtain a clause $\mathcal{C}''$ containing $11$ instances of $\neg p_3$.
\end{example}
 
The map from the set of clauses to the poset $2^n \times \omega^n$ employed in the  proof of Proposition \ref{prop: SCAN_terminates} will be exploited (albeit in a modified form) to prove the termination of MASSA in Proposition \ref{prop: MASSA_termination} below. We conclude with a remark on the difference between our proof and the one in \cite{Goranko:SCAN:Complete}.

\begin{remark}\label{rem:different_from_Goranko}
    The class of analytic inductive formulas and the class of Sahlqvist formulas are contained in the class of inductive formulas, but neither is contained in the other. Translations of Sahlqvist formulas can have arbitrarily complex alternation of quantifiers (e.g.~$p \rightarrow \wbox\wdia\wbox\cdots\wdia\wbox p$), while all analytic inductive formulas are morally equivalent to $\forall\exists$ formulas, since all their branches are good. On the other hand, Sahlqvist formulas only admit unary PIA connectives, meaning that every inductive order $\Omega$ is admissible (cf.~\cite{CoPa12}), and as a consequence, all the SCAN clauses from a sahlqvist formula will contain either a single positive propositional atom or only negative propositional atoms (cf.~\cite[Theorem 2]{Goranko:SCAN:Complete}). In other words, in our case, proving that the unskolemization step successfully terminates is trivial, while verifying that the resolution step ends is more laborious. In contrast, in the sahlqvist case, unskolemization step is not obvious, while it is easy to prove that the resolution step does not loop.
\end{remark}

\section{Soundness and termination of MASSA via a comparison with SCAN} \label{sec:MASSA_soundness_termination}

In this section, we prove that the algorithm MASSA is successful on all definite analytic inductive axioms. In particular, in Subsection \ref{prop: MASSA_termination} we prove that the algorithm always terminates on inductive axioms, while in Subsection \ref{subsec:MASSA_soundness} we show that when \Algname terminates, it gives the correct output.

\subsection{Termination} \label{subsec:MASSA_termination}

Before proving the successful termination of our algorithm, it is helpful to investigate cases where \Algname gets stuck or loops, to better understand the reason of its failure against formulas that are not analytic inductive. We begin observing that step (i) of the algorithm always terminates.

\begin{lemma} \label{lem:MASSA_first_step_successful}
    Given an arbitrary formula $\varphi$ as input, it is always possible to derive $x : \rvarphi \vd x : \bvarphi$ in a way that every label introduced proceeding bottom-up is fresh in the entire derivation.
\end{lemma}
\begin{proof}
    Just follow Lemma \ref{lemma: phi implies phi} and always use new variable names when applying an eigenvariable rule.
\end{proof}

The first reason why \Algname could fail comes from a bad alternation of diamond-like connectives from $\mathcal{F}$ and box-like connectives from $\mathcal{G}$: when a branch of the input formula is not \emph{good} (cf.~\ref{def:good:branch}), the algorithm could get stuck in steps (ii) or (iii). The next examples illustrate this phenomenon.

\begin{example} \label{exe:MASSA_failure_stuck}
    Let us try and run \Algname on the (non inductive and famously non elementary, see \cite{vanBenthem78}) McKinsey formula $\wbox \wdia A \ararr \wdia \wbox A$.
Step (i) produces the leaves
\begin{center}
$xRy, yRz, z:\bA \vd z:\rA \quad\quad xRw, wRt; t:\rA \vd t:\bA,$
\end{center}
but after performing the cut, at step (ii) and (iii) we get stuck:
\begin{center}
\AXC{$xRy,yRz,xRw,wRt,z=t,z:A\vd t:A$}

\UIC{???}

\UIC{$ x: \wbox \wdia A \vd x: \wdia \wbox A$}

\UIC{$\vd x: \wbox \wdia A \ararr \wdia \wbox A$}
\DP
\end{center}
We cannot proceed bottom-up since we do not have the necessary relational information, and we cannot proceed top-down without violating the side conditions of G3K. We fail to reach a merging-point from both directions. Consider now the (Sahlqvist but not analytic) formula $A \ararr \wdia \wbox A$, the first step produces the leaves
\begin{center}
$x:\bA \vd x:\rA \quad\quad xRw, wRt; t:\rA \vd t:\bA.$
\end{center}
Again, after performing the cut, we cannot proceed further:
\begin{center}
\AXC{$xRw,wRt,x=t,x: A\vd t:A$}

\UIC{???}

\UIC{$ x: A \vd x: \wdia \wbox A$}

\UIC{$\vd x: A \ararr \wdia \wbox A$}
\DP
\end{center}

The source of the problem is that in both formulas considered above there is a bad alternation of Skeleton and Pia connectives, as shown in the syntax tree of 
\[\top \leq \wbox \wdia p \rightarrow \wdia \wbox p \equiv \wbox \wdia p \leq \wdia \wbox p\] 
below, where single-circled nodes are PIA nodes and double-circled ones are Skeleton. 

\begin{center}
\begin{tikzpicture}
			\tikzstyle{level 1}=[level distance=1cm, sibling distance=2.5cm]
			\tikzstyle{level 2}=[level distance=1cm, sibling distance=1.5cm]
			\tikzstyle{level 3}=[level distance=1 cm, sibling distance=1.5cm]
			\node[PIA] at (-1.5,0) {$\begin{aligned} +\wbox \end{aligned}$}
			child{node[Ske]{$\begin{aligned} +\wdia \end{aligned}$}
				child{node{$+p$}}
			}
			;
			\node at (0,0) {$\le$}; 
			
			\node[PIA] at (1.5,0) {$\begin{aligned} -\wdia \end{aligned}$}
			child{node[Ske]{$\begin{aligned} -\wbox \end{aligned}$}
				child{node{$-p$}}
			}
			;

			\node at (-2.33,-1.5) {};
			% \draw[help lines] (-4,-4) grid (4,4);
			\end{tikzpicture}
            \end{center}
\end{example}

The next lemma shows that in order not to get stuck in steps (ii) and (iii) of the algorithm, it is enough for every branch of the input inequality to be \emph{good}.

\begin{lemma} \label{lem: Skeleton_terminates}
    If $\varphi \leq \psi$ is a definite analytic-inductive inequality, the following assertions hold:
    \begin{enumerate}
        \item it is always possible to eliminate all the Skeleton connectives of $\varphi \leq \psi$ using only unary rules of the calculus and proceeding up from the sequent $x:\varphi \vd x:\psi$;
        \item it is always possible to build all the PIA subformulas  of $\varphi \leq \psi$ starting from the atomic cuts obtained at the beginning of step (ii) of the algorithm.
    \end{enumerate}
\end{lemma}
\begin{proof}
\begin{enumerate}
    \item We prove our thesis via structural induction, with a slight strengthening of our inductive hypothesis. We want to prove that in every sequent of the type
    \[\mathcal{R}, x_1:\alpha_1,\ldots,x_n:\alpha_n \vdash y_1:\beta_1,\ldots,y_m:\beta_m\]
    where $\mathcal{R}$ contains relational atoms, we can decompose the Skeleton connectives of the $\alpha_i$s and the $\beta_j$s using only unary rules of the calculus. The base case is obvious, since there are no Skeleton connectives to decompose. In the inductive case, it is enough to use one of the rules $\wedge_L$, $\wdia_L$ (if one of the $\alpha_i$s has the shape $\alpha'_1 \wedge \alpha''_i$ or $\wdia \alpha'_i$) $\vee_R$, or $\wbox_R$ (if one of the $\beta_j$s has the shape $\beta'_j \vee \beta''_j$ or $\wbox \beta'_j$). Notice that no $\alpha$ (resp.~$\beta$) has shape $\alpha' \vee \alpha''$ (resp.~$\beta' \wedge \beta''$), since we assumed the input formula to be definite.
    \item Since the subformulas to be constructed are PIA, the only rules for the modal connectives involved in their construction are $\wdia_R$ and $\wbox_L$. It is enough to show that we can always apply them when needed, i.e.~we have the appropriate relational atom $x_iRx_j$. The last statement holds because the starting leaves in step $(ii)$ contain all the relational atoms involved in the construction of the two branches from where the propositional atoms come from.
\end{enumerate}
\end{proof}

The previous lemma shows that step (iii) of \Algname terminates successfully when its input is a definite analytic inductive formula, but while it's true that we can always build all the PIA subformulas (meaning that during step (ii) we do not get stuck), we still have to rule out the possibility that the algorithm continues indefinitely without reaching a point where the sequent contains only some of the maximal PIA subformulas. The next example illustrates this possibility.

\begin{example} \label{exe:MASSA_failure_loop}
    Consider a non-inductive formula where every branch is good but there is no inductive order $\Omega$, like the formula $\wbox ( A\vee B) \rightarrow \wdia (B \wedge A)$ with maximal PIA subformulas $\wbox(A \vee B)$ and $\wdia (A \vee B)$. After the first \Algname step, the leaves obtained by performing the atomic cuts are
\[xRz, xRy, y=z, y:A \vdash z: A\quad xRz, xRy, y=z, y:B \vdash z:B\]
Below we illustrate the attempt to carry out step (ii) of the algorithm, some of the sequents are numbered for later reference. To make the derivation easier to visualize, we write $\mathcal{R}$ in place of $xRz, xRy, y=z$.

\begin{center}
{\small

\AXC{}
\UIC{($1$) $\mathcal{R}, y:A \vdash z:A$}
\AXC{}
\UIC{$\mathcal{R}, y:B \vdash y:B$}
\BIC{($2$) $\mathcal{R},\mathcal{R}, y:A\vee B \vdash z:A,z:B$}
\UIC{($3$) $\mathcal{R},\mathcal{R}, x: \wbox(A\vee B) \vdash z:A,z:B$}

\AXC{}
\UIC{$\mathcal{R},y:A \vdash z:A$}
\BIC{($4$) $\mathcal{R},\mathcal{R},\mathcal{R}, x: \wbox(A\vee B), y:A \vdash z:A,z:A\wedge B$}
\UIC{$\mathcal{R},\mathcal{R},\mathcal{R}, x: \wbox(A\vee B), y:A \vdash z:A,x: \wdia(A\wedge B)$}

\AXC{}
\UIC{$\mathcal{R}, y:B \vdash z: B$}
\BIC{}

\noLine
\UIC{$\vdots\ ???$}
\dashedLine
\UIC{$x: \wbox( A\vee B) \vdash x: \wdia (B \wedge A)$}
\UIC{$\vdash x: \wbox( A\vee B) \rightarrow \wdia (B \wedge A)$}
\DP
}
\end{center}

It is easy to realize that we are going in a loop, unable to build a sequent containing only some of the maximal PIA subformulas. Notice that we are eventually able to construct maximal PIAs, but there will always be some `leftover' formula preventing the successful termination of the algorithm. This example is somehow similar to Example \ref{exe:SCAN_loops} from the previous section.
\end{example}

Before providing the proof of the termination of the algorithm, we need some preliminary definitions. Refer to Example \ref{exe:MASSA_termination_auxiliaries_definitions} to see how the new definitions relate to the previous examples.

\begin{definition} \label{def:MASSA_termination_auxiliaries_definitions}
    Consider an arbitrary inequality $\varphi \leq \psi$ and label every instance of the same atomic proposition uniquely, we denote the labelled atomic propositions \emph{decorated atoms} and we say that $\varphi \leq \psi$ is decorated.  A \emph{couple} is a pairing between different instances of the same atomic proposition coming from a leaf after an atomic cut.
    We now inductively define a special multiset, called the \emph{residue}, associated with a sequent constructed during phase (ii) of the algorithm. For an initial sequent coming from an atomic cut, its residue is the set of all the decorated atoms of the two maximal PIAs containing the couple, not including the two atoms in the couple. From sequents obtained via $\wbox_L$ and $\wdia_R$, their residue is the same as the residue of the premise of the rule. If the sequent is obtained via an application of the rule $\wedge_R$ as shown below,
    \begin{center}
        \AXC{$\Gamma_1 \vdash \Delta_1, x:\alpha$}
        \AXC{$\Gamma_1 \vdash \Delta_2, x:\alpha$}
        \BIC{$\Gamma_1,\Gamma_2 \vdash \Delta_1, \Delta_2, x:\alpha \wedge \beta$}
        \DP
    \end{center}
    its residue is $(\chi_1\setminus\mathrm{Var}(\beta)) \cup (\chi_2 \setminus\mathrm{Var}(\alpha))$, where $\chi_1$ and $\chi_2$ are respectively the residue of the left and the right premise, while $\mathrm{Var}(\varphi)$ is the set of the decorated atoms contained in $\varphi$. The $\vee_L$ case is defined analogously.
\end{definition}

In the special case where in the application of $\wedge_R$ one of the premises is an axiom from an atomic cut, Definition \ref{def:MASSA_termination_auxiliaries_definitions} reduces to

\begin{center}
        \AXC{$\Gamma \vdash \Delta, x:\alpha$}
        \AXC{$\mathcal{R},y:p \vdash x:p$}
        \BIC{$\Gamma, \mathcal{R},y:p \vdash \Delta, x:\alpha \wedge p$}
        \DP
    \end{center}
    and its residue is $(\chi_1\setminus\{p\}) \cup (\chi_2 \setminus\mathrm{Var}(\alpha))$. This special case will be useful in the proof of Proposition \ref{prop: MASSA_termination}.

Not being able to eliminate all the residue coincides with the inability to finish step $(ii)$ of the algorithm. The next example illustrates the connection between the two notions.

\begin{example}\label{exe:MASSA_termination_auxiliaries_definitions}
    Consider the formula  $\wbox(A \vee B) \ararr \wdia (B \wedge A)$ from Example \ref{exe:MASSA_failure_loop}, after the decoration it becomes $\wbox(A_1 \vee B_1) \ararr \wdia (B_2 \wedge A_2)$.
    The leaves obtained after the atomic cuts are
    \[xRz, xRy, y=z, y:A_1 \vdash z: A_2\quad xRz, xRy, y=z, y:B_1 \vdash z:B_2,\]
    meaning that the couples of the formula are exactly $(A_1,A_2)$ and $(B_1,B_2)$. Notice that the instances composing the couple always have opposite polarities and come from different PIAs. The residue of the sequents ($1$), ($2$), ($3$), and ($4$) are $\{B_1, B_2\}$, $\{A_2, B_2\}$, $\{A_2, B_2\}$, and $\{B_1,B_2\}$, respectively. Now it is clear that the second phase does not terminate, because the residue does not decrease, but it is looping.
\end{example}

We finally prove that \Algname does not loop when its input is an analytic inductive inequality, preventing situations like the one in Example \ref{exe:MASSA_failure_loop} from occurring.

\begin{proposition} \label{prop: MASSA_termination}
    The algorithm MASSA terminates on definite analytic inductive inequalities.
\end{proposition}
\begin{proof}
    To start building the maximal PIA subformulas from the sequents obtained after performing the atomic cuts, we choose an arbitrary couple (cf.~Definition \ref{def:MASSA_termination_auxiliaries_definitions}). We show that whatever our starting couple is, we are always able to zero its residual. 
    The starting residue contains at most one critical instance (the two instances of the couple have opposite polarity), meaning that we can define a map from residues to the lexicographically ordered poset $2^n \times \omega^n$ ($n$ being the number of atomic propositions in the inequality) as follows:
    \[\{*p_i, p_1,\ldots,p_n\}\mapsto 0^{n-i}\times 1\times 0^{i-1}\times (m_1\ldots,m_n),\]
    where the asterisk denotes the only critical instance and $m_j$ is the multiplicity of $p_j$ in the residue. Notice that this definition differs slightly from the one in Proposition \ref{prop: SCAN_terminates}, in the sense that when the critical instance of $p_i$ is smaller with respect to $\Omega$, now $p_i$ is larger with respect to the order induced by the map. We denote the value of the map as the \emph{size} of the residue. The sizes are well-ordered, meaning that to prove the termination of MASSA is sufficient to argue that every step shrinks the residue, whereas with `step' we mean removing decorated atoms from the residue adding new couples via binary rules, updating the residue appropriately as explained in Definition \ref{def:MASSA_termination_auxiliaries_definitions}. We distinguish two cases.

    If the chosen decorated atom to remove is not critical, its partner (i.e.~the second instance of the couple) will be critical, meaning that in the new residue only non-critical decorated atoms are added, that are strictly smaller than the critical instance that got removed.

    If the chosen element to remove from the residue is critical, its partner will be non-critical. If the new residue has no critical instances, then the size decreases, whereas if there is one (and necessarily not more than one) critical instance, it will be bigger with respect to the order $\Omega$, meaning that the size of the resulting residue will still be smaller.
\end{proof}

\subsection{Soundness} \label{subsec:MASSA_soundness}

Having shown that MASS terminates on definite analytic inductive formulas, in this subsection we prove that its output is the correct one. To do this, we will show that the first-order correspondent of MASSA is equivalent to the one produced by SCAN.

\begin{proposition} \label{prop: MASSA_soundness_via_SCAN}
    The algorithm MASSA gives the same output as SCAN on definite analytic inductive axioms.
\end{proposition}
\begin{proof}
We recall from Remark \ref{rem:SCAN_MASSA_comparison} that on an arbitrary definite analytic inductive formula $\alpha$, SCAN starts with clauses
 \[\{x_1Ry_1\},\ldots,\{x_nRy_n\},\]
    \[\{\neg z_{1,1}Rw_{1,1},\ldots,\neg z_{1,k_1}Rw_{1,k_1},*P_{1,1}(t_{1,1}),\ldots,*P_{1,k'_1}(t_{1,k'_1})\},\]
    \[\vdots\]
    \[\{\neg z_{m,1}Rw_{m,1},\ldots,\neg z_{m,k_m}Rw_{m,k_m},*P_{m,1}(t_{m,1}),\ldots,*P_{m,k'_m}(t_{m,k'_m})\}.\]
Where all the variables are quantified as explained previously in the remark above. After applying resolution and deleting all the clauses containing second-order predicates, we obtain something of the form
 \[\{x_1Ry_1\},\ldots,\{x_nRy_n\},\]
    \[\{\neg j_{1,1}Rh_{1,1},\ldots,\neg j_{1,l_1}Rh_{1,l_1},s_{1,1}\neq v_{1,1},\ldots,s_{1,l'_1}\neq v_{1,l'_1}\},\]
    \[\vdots\]
    \[\{\neg j_{m',1}Rh_{m',1},\ldots,\neg j_{m',l_{m'}}Rh_{m',l_{m'}},s_{m',1}\neq v_{m',1},\ldots,s_{m',l'_{m'}}\neq v_{m',l'_{m'}}\}.\]
Where all the variables different from $x_1,y_1,\ldots,x_n,y_n$ are universally quantified. If we denote the last $m'$ clauses above with $\mathcal{B}_1,\ldots,\mathcal{B}_{m'}$ respectively, once we translate the output into a first-order sentence and negate it, we get
\begin{equation} \label{eq:SCAN_output_geometric}
\forall \overline{x},\overline{y}(x_1Ry_1 \wedge \cdots \wedge x_nRy_n \rightarrow \exists \overline{z_1}\neg\mathcal{B}_1 \vee \cdots\vee\overline{z_{m'}}\neg\mathcal{B}_{m'}),
\end{equation}
where $\overline{z_i}$ are all the variables in $\mathcal{B}_i$ different from $\overline{x}$ and $\overline{y}$, while $\neg \mathcal{B}_i$ is the conjunction of all the negated literals in $\mathcal{B}_i$. Notice that \ref{eq:SCAN_output_geometric} is a geometric formula, and it is clear that if we run \Algname on $\alpha$, we get at least the same tail of universal quantifiers and the same antecedent, from the decomposition of $\alpha$'s Skeleton connectives during phase $(iii)$. It remains to show that every $\neg \mathcal{B}_i$ corresponds to the relational atoms in one of the premises of the merging point created during phase $(ii)$. The last statement follows from the observation that all the inequalities $s\neq v$ come from the first round of resolution (where all the premises of the resolution step are input clauses) and they correspond to the atomic cuts; furthermore, in all the remaining rounds of resolution we just merge first-order literals without creating new relational information, and they correspond to building the maximal PIA subformulas of $\alpha$ by adding couples (cf.Definition \ref{def:MASSA_termination_auxiliaries_definitions}).
   
\end{proof}

We are now ready to state the main result of this section, leveraging the soundness of the algorithm SCAN and the propositions proven above.

\begin{theorem} \label{thm: MASSA_successful}
    The algorithm MASSA is successful on definite analytic inductive axioms.
\end{theorem}
\begin{proof}
    We proved in Proposition \ref{prop: MASSA_termination} that MASSA terminates on definite analytic inductive axioms, while in Proposition \ref{prop: MASSA_soundness_via_SCAN} we proved that on definite analytic inductive axioms, its output coincides with the output of the algorithm SCAN. Since SCAN is complete with respect to analytic inductive formulas (cf.~Proposition \ref{prop: SCAN_terminates}), this implies that MASSA yields the correct first-order correspondent and, accordingly, the correct equivalent analytic rule.
\end{proof}

\section{Extending MASSA} \label{sec:extending_MASSA}

In this section we show that our approach can be readily generalized to much broader contexts. Without reporting detailed proofs, we will illustrate through examples how \Algname can be modified to work even with inductive and not necessarily analytic inductive formulas. We will show how to generalize the algorithm in the context of first-order modal logic and in the context of distributive lattice expansion logics (DLE).

To make extensions to more general cases work, MASSA must sometimes be modified appropriately. In this section, in lieu of giving a different definition of the algorithm in each case, we will describe from time to time the modifications needed as the examples unfold.

\subsection{Extending MASSA to inductive axioms} \label{Extending_MASSA_inductive}

As previously mentioned in Remark \ref{rem:different_from_Goranko}, the proof that SCAN is successful on the class of Sahlqvist axioms \cite{Goranko:SCAN:Complete} gives us a strategy for coping with the unksolemization step when using SCAN on arbitrary inductive axioms. Together with our argument for the termination of the resolution phase, we conjuncture we could find a proof that SCAN is successful on the whole class of inductive axioms. In this section, we give some examples of how MASSA can be used to extract the first-order correspondents together with their equivalent analytic systems of rules (in the sense of \cite{Negri14}) of the whole class of inductive formulas in the classical setting.

As a first example, consider the axiom $p \ararr \Box \Box \wdia \Box \wdia p$. The axiom is definite and Sahlqvist (hence inductive), but not analytic inductive. The first step of \Algname produces the leaves

\[x: \bp \vd x: \rp \quad xRy_1, y_1Ry_2, y_2Ry_3, y_3Ry_4, y_4Ry_5, x:\rp \vd y_5:\bp\]

Starting from the leaf $xRy_1, y_1Ry_2, y_2Ry_3, y_3Ry_4, y_4Ry_5, x=y_5, x:p \vd y_5:p$, for every layer of Skeleton/PIA connectives we proceed top down as follows, starting from the first innermost layer. 

\begin{itemize}
    \item We build the $n$-th layer of PIA connectives, after that we mark the relational atoms involved in the construction and we draw the $n$-th merging point. In the case $n=1$, the equality atoms coming from the atomic cuts start already marked.
    \item We copy the labelled formulas and the un-marked relational atoms below the merging point and we build the $n$-th layer of Skeleton connectives.
    \item If we are not finished, we go back to the first point and build the $n+1$-th layer.  
\end{itemize}

\begin{center}
{\small
\AXC{$xRy_1, y_1Ry_2, y_2Ry_3, y_3Ry_4, y_4Ry_5, \textcolor{red}{x=y_5}, x:p \vd y_5:p$}
\UIC{$xRy_1, y_1Ry_2, y_2Ry_3, y_3Ry_4, \textcolor{red}{y_4Ry_5}, \textcolor{red}{x=y_5}, x:p \vd y_4:\wdia p$}
\dashedLine
\UIC{$xRy_1, y_1Ry_2, y_2Ry_3, y_3Ry_4, x:p \vd y_4:\wdia p$}
\UIC{$xRy_1, y_1Ry_2, y_2Ry_3, x:p \vd y_3:\Box\wdia p$}
\doubleLine
\UIC{$xRy_1, y_1Ry_2, y_2Ry_3, x:p \vd y_3:\Box\wdia p$}
\UIC{$xRy_1, y_1Ry_2, \textcolor{red}{y_2Ry_3}, x:p \vd y_2:\wdia\Box\wdia p$}
\dashedLine
\UIC{$xRy_1, y_1Ry_2, x: p \vd y_2: \wdia \Box \wdia p$}
\UIC{$xRy_1, x: p \vd y_1: \Box \wdia \Box \wdia p$}
\UIC{$x: p \vd x: \Box \Box \wdia \Box \wdia p$}
\UIC{$\vd x: p \ararr \Box \Box \wdia \Box \wdia p$}
\DP
}
\end{center}

Notice how this time we got two merging points, since the formula $ p \ararr \Box \Box \wdia \Box \wdia p$ has two layers of Skeleton/PIA alternations. In the derivation of the axiom above, a double line separates the two layers. The relational atoms involved in the construction of the PIA connectives are marked in red. The first order correspondent extracted from the two merging point is the following depth-$2$ general geometric axiom (cf.~\cite{Negri14}): 

\[\forall x \forall y_1 \forall y_2 (xRy_1 \wedge y_1Ry_2 \ararr \exists y_3(y_2Ry_3 \wedge \forall y_4(y_3Ry_4 \ararr \exists y_5(y_4Ry_5 \wedge x=y_5)))),\]
with the corresponding system of rules
\[ \begin{cases}
\AXC{$xRy_1, y_1Ry_2, y_2Ry_3, \Gamma \vd \Delta$}
\UIC{$xRy_1, y_1Ry_2, \Gamma \vd \Delta$}
\DP \\ \\
\AXC{$xRy_1, y_1Ry_2, y_2Ry_3, y_3Ry_4, y_4Ry_5, x=y_5, \Gamma \vd \Delta$}
\UIC{$xRy_1, y_1Ry_2, y_2Ry_3, y_3Ry_4, \Gamma \vd \Delta$}
\DP
\end{cases} \] 

As a second example, consider now the definite inductive (but not sahlqvist) axiom $\wdia p \wedge q \ararr \wdia \Box (p \vee \wdia q)$. The leaves obtained after the atomic cut are

\[xRz,zRw,wRt,t=x,x:q \vd t:q\quad xRz,zRw,xRy,w=y,y:p \vd w:p\]

We follow the same procedure outlined above.

\begin{center}
{\small

\AXC{$xRz,zRw,wRt,\textcolor{red}{t=x},x:q \vd t:q$}
\UIC{$xRz,zRw,\textcolor{red}{wRt},\textcolor{red}{t=x},x:q \vd w:\wdia q$}

\AXC{$xRz,zRw,xRy,\textcolor{red}{w=y},y:p \vd w:p$}

\dashedLine
\BIC{$xRy, xRz, zRw,  y: p, x: q \vd w: p, w: \wdia q$}
\UIC{$xRy, xRz, zRw,  y: p, x: q \vd w: p \vee \wdia q$}
\UIC{$xRy, xRz,  y: p, x: q \vd z: \Box (p \vee \wdia q)$}
\doubleLine
\UIC{$xRy, xRz,  y: p, x: q \vd z: \Box (p \vee \wdia q)$}
\UIC{$xRy, \textcolor{red}{xRz},  y: p, x: q \vd x: \wdia \Box (p \vee \wdia q)$}

\dashedLine
\UIC{$xRy,  y: p, x: q \vd x: \wdia \Box (p \vee \wdia q)$}
\UIC{$ x: \wdia p, x: q \vd x: \wdia \Box (p \vee \wdia q)$}
\UIC{$ x: \wdia p \wedge q \vd x: \wdia \Box (p \vee \wdia q)$}
\UIC{$\vd x: \wdia p \wedge q \ararr \wdia \Box (p \vee \wdia q)$}
\DP
}
\end{center}

Again, we extract the first-order correspondent from the merging point

\[\forall x \forall y(xRy \ararr \exists z(xRz \wedge \forall w(zRw \ararr \exists t(wRt \wedge t=x) \vee w=y))),\]
and we write the corresponding system of rules.
\[ \begin{cases}
\AXC{$xRy, xRz,  \Gamma \vd \Delta$}
\UIC{$xRy,  \Gamma \vd \Delta$}
\DP \\ \\
\AXC{$xRz,zRw,wRt,t=x,\Gamma \vd\Delta$}
\AXC{$xRz,zRw,xRy,w=y,\Gamma \vd \Delta$}
\BIC{$xRy, xRz, zRw, \Gamma \vd \Delta$}
\DP
\end{cases} \] 

\subsection{Extending MASSA to the quantified setting} \label{subsec:estending_MASSA_quantified}

The results of the previous sections can be generalized to the quantificational setting with minimal effort. We provide some interesting examples in the G3-style labelled calculus for classical first-order modal logic described in \cite{negri2011proof}, where the new rules for the quantifiers are the following.

\begin{center}
    \begin{tabular}{rl}
\AXC{$a \in D(w), \Gamma \vd \Delta, w: A(a/x)$}
\LL{$\forall_L$}
\UIC{$\Gamma\vd \Delta, w: \forall x A$}
\DP
 & 
\AXC{$w: A(a/x), w: \forall x A, a \in D(w), \Gamma \vd \Delta$}
\RL{$\forall_R$}
\UIC{$w: \forall x A, a \in D(w),\Gamma \vd \Delta$}
\DP
 \\ \\
 \AXC{$a \in D(w), w: A(a/x),\Gamma \vd \Delta$}
\LL{$\exists_L$}
\UIC{$w: \exists x A,\Gamma\vd \Delta$}
\DP
 & 
\AXC{$a \in D(w),\Gamma \vd \Delta, w:\exists x A, w: A(a/x)$}
\RL{$\exists_R$}
\UIC{$a \in D(w),\Gamma \vd \Delta,w: \exists x A$}
\DP
\end{tabular}
\end{center}

From the point of view of our syntactical analysis, universal (resp.~existential) quantifiers are treated like box (resp.~diamond) connectives, meaning that $-\forall$ (resp.~$+\exists$) is a Skeleton connective and $+\forall$ (resp.~$-\exists$) is a PIA connective. Consider the analytic inductive formula $\wdia \forall x A \ararr \forall x \wdia A$. For the first phase, we carry out the derivation of the identity sequent as follows:

\begin{center}
{\footnotesize
\AXC{$(\pi_1)$}
\AXC{$(\pi_2)$}
\LL{$\ararr_L$}
\BIC{$ \ w: \wdia \forall x \rA(x) \ararr \forall x \wdia \rA(x), w: \wdia \forall x \bA(x) \vd w: \forall x \wdia \bA(x)$}
\RL{$\ararr_R$}
\UIC{$ \ w: \wdia \forall x \rA(x) \ararr \forall x \wdia \rA(x) \vd w: \wdia \forall x \bA(x) \ararr \forall x \wdia \bA(x)$}
\DP
}
\end{center}
Where $(\pi_1)$ is the derivation
\begin{center}
\AXC{$ \ $}
\RL{Id$_{u:A(a/x)}$}
\UIC{$wRu,a\in D(u), u: \bA(a/x)\vd u: \rA(a/x)$}
\LL{$\forall_L$}
\UIC{$wRu,a\in D(u), u: \forall x \bA(x)\vd u: \rA(a/x)$}
\RL{$\forall_R$}
\UIC{$wRu, u: \forall x \bA(x)\vd u: \forall x \rA(x)$}
\RL{$\wdia_R$}
\UIC{$wRu, u: \forall x \bA(x)\vd w:\wdia \forall x \rA(x)$}
\LL{$\wdia_L$}
\UIC{$w:\wdia \forall x \bA(x)\vd w:\wdia \forall x \rA(x)$}
\DP
\end{center}
and $(\pi_2)$ is the derivation
\begin{center}
    \AXC{$ \ $}
\RL{Id$_{v:A(b/x)}$}
\UIC{$b \in D(w), wRv, v: \rA(b/x) \vd v: \bA(b/x)$}
\RL{$\wdia_R$}
\UIC{$b \in D(w), wRv, v: \rA(b/x) \vd w:\wdia \bA(b/x)$}
\LL{$\wdia_L$}
\UIC{$b \in D(w),w: \wdia \rA(b/x) \vd w:\wdia \bA(b/x)$}
\LL{$\forall_L$}
\UIC{$b \in D(w),w: \forall x \wdia \rA(x) \vd w:\wdia \bA(b/x)$}
\RL{$\forall_R$}
\UIC{$w: \forall x \wdia \rA(x) \vd w: \forall x \wdia \bA(x)$}
\DP
\end{center}

The remaining phases are carried out as follows:

\begin{center}
{\small
\AXC{$ wRu, wRv, a \in D(u),b\in D(w), u=v, a=b, u:A(a/x) \vd v: A(b/x) $}
\UIC{$ wRu, wRv, a \in D(u),b\in D(w), u=v, a=b, u:A(a/x) \vd w: \wdia A(b/x) $}
\UIC{$ wRu, wRv, a \in D(u),b\in D(w), u=v, a=b, u:\forall x A(x) \vd w: \wdia A(b/x) $}

\dashedLine
\UIC{$wRu, b\in D(w), u: \forall x A(x)\vd w: \wdia A(b/x)$}
\UIC{$wRu, u: \forall x A(x)\vd w:\forall x \wdia A(x)$}
\UIC{$ w:\wdia \forall x A(x)\vd w:\forall x \wdia A(x)$}
\UIC{$\vd w:\wdia \forall x A(x)\ararr\forall x \wdia A(x)$}
\DP
}
\end{center}

As expected, the first-order correspondent of $\wdia \forall x A \rightarrow \forall x \wdia A$ is

\[\forall w \forall u \forall b (wRu \wedge b\in D(w) \ararr \exists v \exists a (u=v \wedge a=b \wedge wRv \wedge a\in D(u))),\]
once simplified it becomes
\[\forall w \forall u \forall b (wRu \wedge b\in D(w) \ararr \wedge b\in D(u)).\]

Consider now a `Barcan-style' seriality $\exists x \Box A \ararr \wdia \exists x A$ and omit the trivial first phase. The two endsequents produced are:
\[a \in D(w), wRv, v:\bA(a/x) \vd v:\rA(a/x),\]
\[wRu, b\in D(u), u:\rA(b/x)\vd u:\bA(b/x)\]

We cut the two leaves and produce the merging point as follows:

\begin{center}
{\small
\AXC{$ wRu, wRv, a \in D(w), b \in D(u), u=v, a=b, v:A(a/x)\vd u: A(b/x) $}
\UIC{$ wRu, wRv, a \in D(w), b \in D(u), u=v, a=b, w:\Box A(a/x)\vd u: A(b/x) $}
\UIC{$ wRu, wRv, a \in D(w), b \in D(u), u=v, a=b, w:\Box A(a/x)\vd u: \exists x A(x) $}
\UIC{$ wRu, wRv, a \in D(w), b \in D(u), u=v, a=b, w:\Box A(a/x)\vd w: \wdia \exists x A(x) $}

\dashedLine
\UIC{$ a\in D(w), w: \Box A(a/x) \vd w: \wdia \exists x A(x)$}
\UIC{$ w:\exists x \Box A(x) \vd w: \wdia \exists x A(x)$}
\UIC{$\vd w:\exists x \Box A(x) \ararr \wdia \exists x A(x)$}
\DP
}
\end{center}

We finally extract the first-order correspondent from the merging point:

\[\forall w \forall a (a \in D(w) \ararr \exists u \exists v \exists b(wRu, wRv, b\in D(u), u=v, a=b)),\]
or equivalently
\[\forall w \forall a (a \in D(w) \ararr \exists u (wRu \wedge a\in D(u))).\]

%% file: paper_MASSA/Conclusions.tex
\section{Conclusions and future work}
\label{sec: conclusions}

\paragraph{Main contributions.} In this article we presented MASSA, an algorithm that exploits the properties of G3-style calculi to generate analytical rules equivalent to certain axioms belonging to an appropriate class of first-order definable formulas. We have shown that the algorithm terminates and is correct when its input belongs to the class of definite analytic inductive formulas. MASSA's proof of correctness is based on the correctness of the SCAN algorithm, and to achieve the goal we proved an interesting result in itself: SCAN is complete with respect to the class of analytic inductive formulas. We illustrated how to extend our approach to the broader class of inductive axioms, moving from rules to systems of rules, then further explained how to generalize the algorithm to the quantified setting.

\paragraph{Related work.} 
The results in the present paper pertain to a larger line of research in structural proof theory focusing on the uniform generation of analytic rules for classes of axiomatic extensions in different (nonclassical) logics, which includes  e.g., \cite{Sim94,Vig00,NegVonPla98,Neg03,negri2005proof} in the context of sequent and labelled calculi, \cite{ciabattoni2008axioms,lahav2013frame,lellmann2014axioms} in the context of sequent and hypersequent calculi, and \cite{Kracht,CiRa14,GMPTZ} in the context of (proper) display calculi. We refer to \cite{ChnGrePalTzi21} for an overview of this literature. 

\paragraph{Future prospects.} We plan to exploit the algorithm \Algname as a tool to obtain new results in proof analysis and correspondence theory. We hope that this new proof-theoretic approach to correspondence theory will shed new light on the theoretical foundations of the generation of analytic rules for axiomatic extensions in nonclassical logic.